\documentclass[10pt]{amsart}

\usepackage{amssymb}
\usepackage{amsthm}
\usepackage{amsmath}

\usepackage[usenames]{color}
\definecolor{ANDREW}{RGB}{255,127,0}

\usepackage{tikz}
\usetikzlibrary{arrows,calc}

\usepackage[foot]{amsaddr}

\usepackage{hyperref}

\theoremstyle{plain}
\newtheorem{proposition}{Proposition}[section]
\newtheorem{theorem}[proposition]{Theorem}
\newtheorem{lemma}[proposition]{Lemma}
\newtheorem{corollary}[proposition]{Corollary}
\theoremstyle{definition}

\newtheorem{definition}[proposition]{Definition}

\newtheorem*{assumption}{Assumption}
\newtheorem*{addassumption}{Additional Assumption}
\theoremstyle{remark}
\newtheorem{remark}[proposition]{Remark}

\DeclareMathOperator{\Aut}{Aut}

\DeclareMathOperator{\SL}{SL}
\DeclareMathOperator{\GL}{GL}

\DeclareMathOperator{\SO}{SO}
\DeclareMathOperator{\SU}{SU}

\DeclareMathOperator{\PGL}{PGL}

\DeclareMathOperator{\Hol}{Hol}

\DeclareMathOperator{\Span}{Span}

\DeclareMathOperator{\id}{id}

\DeclareMathOperator{\Sp}{Sp}
\DeclareMathOperator{\Ad}{Ad}
\DeclareMathOperator{\PU}{PU}

\DeclareMathOperator{\Ec}{\mathcal{E}}

\DeclareMathOperator{\Gc}{\mathcal{G}}

\DeclareMathOperator{\Lc}{\mathcal{L}}

\DeclareMathOperator{\Pc}{\mathcal{P}}

\DeclareMathOperator{\Bb}{\mathbb{B}}
\DeclareMathOperator{\Cb}{\mathbb{C}}
\DeclareMathOperator{\Db}{\mathbb{D}}

\DeclareMathOperator{\Nb}{\mathbb{N}}

\DeclareMathOperator{\Rb}{\mathbb{R}}

\DeclareMathOperator{\Zb}{\mathbb{Z}}

\DeclareMathOperator{\aL}{\mathfrak{a}}
\DeclareMathOperator{\gL}{\mathfrak{g}}

\newcommand{\abs}[1]{\left|#1\right|}

\newcommand{\norm}[1]{\left\|#1\right\|}
\newcommand{\wt}[1]{\widetilde{#1}}

\newcommand{\ip}[1]{\left\langle #1\right\rangle}

\begin{document}

\title[The automorphism group and limit set of a bounded domain]{The automorphism group and limit set of a bounded domain I: the finite type case}
\author{Andrew Zimmer}\address{Department of Mathematics, Louisiana State University, Baton Rouge, LA 70803}
\email{amzimmer@lsu.edu}
\date{\today}

\maketitle

\begin{abstract} For bounded pseudoconvex domains with finite type we give a precise description of the automorphism group: if an orbit of the automorphism group accumulates on at least two different points of the boundary, then the automorphism group has finitely many components and is the almost direct product of a compact group and connected Lie group locally isomorphic to ${ \rm Aut}(\mathbb{B}_k)$. Further, the limit set is a smooth submanifold diffeomorphic to the sphere of dimension $2k-1$. As applications we prove a new finite jet determination theorem and a Tits alternative theorem. The geometry of the Kobayashi metric plays an important role in the paper. 
\end{abstract}

\section{Introduction} 

Given a domain $\Omega \subset \Cb^d$, let $\Aut(\Omega)$ denote the automorphism group of $\Omega$, that is the group of biholomorphic maps $\Omega \rightarrow \Omega$. The group $\Aut(\Omega)$ is a topological group when endowed with the compact-open topology and when $\Omega$ is bounded H. Cartan proved that $\Aut(\Omega)$ is a Lie group. We will let $\Aut_0(\Omega)$ denote the connected component of the identity in $\Aut(\Omega)$.  The \emph{limit set of $\Omega$}, denoted $\Lc(\Omega)$, is the set of points $x \in \partial \Omega$ where there exists some $z \in \Omega$ and a sequence $\varphi_n \in \Aut(\Omega)$ such that $\varphi_n(z) \rightarrow x$. When $\Omega$ is bounded, $\Aut(\Omega)$ acts properly on $\Omega$. Hence for bounded domains, $\Lc(\Omega)$ is non-empty if and only if $\Aut(\Omega)$ is non-compact. 

This is the first of a series of papers studying the group $\Aut(\Omega)$ and the set $\Lc(\Omega)$. A well understood family of examples are the so-called \emph{generalized ellipses}: 
\begin{align*}
\Ec_{m_1, \dots, m_d} = \left\{ (z_1,\dots, z_d)  \in \Cb^{d} : \abs{z_1}^{2m_1} +\dots +\abs{z_d}^{2m_d} < 1 \right\}
\end{align*}
where $m_1, \dots, m_d \in \Nb$. Webster~\cite{W1979} has given an explicit description of $\Aut(\Ec_{m_1,\dots, m_d})$ (also see~\cite{N1968,Lan1984}). First, by permuting coordinates, we may assume that 
\begin{align*}
m_1=\dots=m_k=1 < m_{k+1} \leq \dots \leq m_d.
\end{align*}
Then if $\Bb_k \subset \Cb^k$ is the unit ball and $\phi \in \Aut(\Bb_k)$, define a rational function $S_{\phi}:  \Cb^k \rightarrow \Cb$ by
\begin{align*}
S_{\phi}(z)  = \frac{1-\abs{  \phi^{-1}(0)}^2}{\left( 1-\ip{z,   \phi^{-1}(0)} \right)^2}.
\end{align*}
Also given $z=(z_1,\dots, z_d) \in \Cb^d$, let $z^k=(z_1,\dots, z_k) \in \Cb^k$.  Then Webster~\cite{W1979} showed that $\Aut(\Ec_{m_1,\dots, m_d})$ has finitely many components and $\varphi \in \Aut_0(\Ec_{m_1,\dots, m_d})$ if and only if
\begin{align*}
\varphi(z) = \left(\phi\left(z^k\right), \ z_{k+1}e^{i\theta_{k+1}}S_{\phi}\left(z^k\right)^{1/2m_{k+1}}, \dots, \ z_{d}e^{i\theta_d}S_{\phi}\left(z^k\right)^{1/2m_d} \right)
\end{align*}
for some $\phi \in \Aut(\Bb_k)$  and $\theta_{k+1},\dots, \theta_d \in \Rb$. So, if we let $N \leq \Aut(\Ec_{m_1, \dots, m_d})$ denote the subgroup of elements of the form
\begin{align*}
\varphi(z) = \left(z_1,\dots, z_k, \ z_{k+1}e^{i\theta_{k+1}} , \dots, \ z_{d}e^{i\theta_{d}}\right),
\end{align*}
then $N \leq \Aut_0(\Omega)$ is a normal compact subgroup and the quotient $\Aut_0(\Omega)/N$ is isomorphic to $\Aut(\Bb_k)$.

Webster's explicit description of the automorphism group also implies the following: if $e_1, \dots, e_d$ is the standard basis of $\Cb^d$, then
\begin{align*}
\Lc\left(\Ec_{m_1, \dots, m_d}\right) = \partial \Ec_{m_1, \dots, m_d} \cap \Span_{\Cb}\{e_1, \dots, e_k\}.
\end{align*}
So for generalized ellipses the limit set is always a smooth submanifold of the boundary which is diffeomorphic to an odd dimensional sphere.  

The main result of this paper shows that all these properties of generalized ellipses extend to pseudoconvex domains with finite type  (see Definition~\ref{defn:FT} below). Before stating the result we introduce a special class of algebraic domains. 
 
 We say a real polynomial $p: \Cb^{d} \rightarrow \Rb$ is a \emph{weighted homogeneous polynomial} if there exists positive integers $m_1, \dots, m_{d}$ such that 
 \begin{align*}
 p(t^{1/m_1}z_1, \dots, t^{1/m_{d}} z_{d}) = t p(z_1,\dots, z_d)
 \end{align*}
 for all $t > 0$ and $z_1,\dots, z_{d} \in \Cb$. 

\begin{definition} A domain $\Pc$ is called a \emph{weighted homogeneous polynomial domain} if 
\begin{align*}
\Pc = \left\{ (w,z) \in \Cb \times \Cb^{d-1} : { \rm Im}(w) >  p(z) \right\}
\end{align*}
where $p:\Cb^d \rightarrow \Rb$ is a weighted homogeneous polynomial.
\end{definition}

Notice that a weighted homogeneous polynomial domain always has non-trivial automorphisms, namely real translations in the first variable and a dilation coming from the fact that $p$ is weighted homogeneous. 

Also, given a group $G$ and subgroups $G_1, \dots, G_n \leq G$ we say that $G$ is the \emph{almost direct product of $G_1, \dots, G_n$} if $G=G_1\cdots G_n$ and distinct pairs of $G_1,\dots, G_n$ commute and have finite intersection. With this terminology we will prove the following. 
 
 \begin{theorem}\label{thm:main_ft} Suppose $\Omega \subset \Cb^d$ is a bounded pseudoconvex domain with $C^\infty$ boundary and finite type. Assume $\Lc(\Omega)$ contains at least two distinct points. Then:
 \begin{enumerate}
 \item $\Omega$ is biholomorphic to a weighted homogeneous polynomial domain.  
 \item $\Aut(\Omega)$ has finitely many connected components. 
 \item $\Aut(\Omega)$ is the almost direct product of closed subgroups $G$ and $N$ where 
 \begin{enumerate}
 \item $N$ is compact,
 \item $G$ is a connected Lie group with finite center and there exists an isomorphism $\rho:G/Z(G)\rightarrow \Aut(\Bb_k)$ for some $k \geq 1$.
 \end{enumerate}
   \item $\Lc(\Omega)$ is a smooth submanifold of $\partial \Omega$ and there exists an $\rho$-equivariant diffeomorphism $F:\Lc(\Omega) \rightarrow \partial \Bb_k$. In particular, $\Lc(\Omega)$ is an odd dimensional sphere and so either 
 \begin{enumerate}
 \item $\dim \Lc(\Omega) \leq \dim \partial \Omega - 2$ or
 \item $\Lc(\Omega) = \partial \Omega$ and $\Omega$ is biholomorphic to the unit ball. 
 \end{enumerate}
 \end{enumerate}
 \end{theorem}
 
 \begin{remark} \ \begin{enumerate} 
 \item The definition of finite type is given in Section~\ref{sec:FT} below. 
\item We will use work of S.Y. Kim~\cite{K2010} to show that $\Omega$ is biholomorphic to a weighted homogeneous polynomial domain.  
 \item The proof of Theorem~\ref{thm:main_ft} uses Catlin's deep work on finite type domains \cite{C1987}. In the less general case of pseudoconvex domains with real analytic boundary, Catlin's results are not needed and instead one could use results of Kohn~\cite{Kohn1977} and Diederich and Forn{\ae}ss~\cite{DF1978}. 
   \item A theorem of Griffiths~\cite{G1971} implies that there exists examples of domains $\Omega \subset \Cb^2$ where $\Aut(\Omega)$ is infinite, discrete, and the quotient $\Aut(\Omega) \backslash \Omega$ is compact (see ~\cite{GR2015} for details). The last condition implies that $\Lc(\Omega) = \partial \Omega$.  These examples never have $C^2$ boundary by a theorem of Rosay~\cite{R1979}.
   \end{enumerate}
  \end{remark}
 
Theorem~\ref{thm:main_ft} provides a precise description of the algebraic structure of $\Aut(\Omega)$ and its action on $\partial \Omega$. Using this description we will prove two corollaries. 
 
The first result involves determining an automorphism from its $k$-jet. In particular, given a diffeomorphism $f:M \rightarrow M$ of a manifold $M$, let $j_k(M,f,x)$ denote the $k$-jet of $f$ at $x \in M$. Then let ${ \rm Jet}_k(M,x)$ denote the set of all $k$-jets at $x$. A well-known problem in CR-geometry is to prove that a CR-automorphism (under certain non-degeneracy conditions) is determined by its $k$-jet for some $k>0$, see for instance~\cite{BK1994, BER2000, ELZ2003,LM2007,LM2007b,J2009,BC2014}. 

By theorems of Bell and Ligocka~\cite{BL1980} and Catlin~\cite{C1987} every biholomorphism of a bounded pseudoconvex domain with finite type extends to a CR-automorphism of its boundary (see Theorem~\ref{thm:Bell1} below). In particular, if $\Omega \subset \Cb^d$ is a bounded pseudoconvex domain with finite type, $\varphi \in \Aut(\Omega)$, and $x \in \partial \Omega$, then 
\begin{align*}
j_k(\partial\Omega, \varphi, x) \in { \rm Jet}_k(\partial \Omega, x)
\end{align*}  
is well defined for any $k \geq 0$. For these extensions we prove the following finite jet determination theorem. 
 
\begin{corollary} (See Section~\ref{sec:finite_jet})\label{cor:finite_jet} Suppose $\Omega \subset \Cb^d$ is a bounded pseudoconvex domain with $C^\infty$ boundary and finite type. Assume $\Lc(\Omega)$ contains at least two distinct points.  Then:
\begin{enumerate}
\item  For any $x \in \Lc(\Omega)$ the map 
\begin{align*}
g \in \Aut(\Omega) \rightarrow j_2(\partial \Omega, g,x) \in { \rm Jet}_2(\partial \Omega, x)
\end{align*}
is injective. 
\item For any $x \in \partial \Omega \backslash \Lc(\Omega)$ the map
\begin{align*}
g \in \Aut(\Omega) \rightarrow j_1(\partial\Omega, g,x)\in { \rm Jet}_1(\partial \Omega, x)
\end{align*}
is injective. 
\item If $N$ is the subgroup from Theorem~\ref{thm:main_ft}, then for any $x \in \partial \Omega$ the map
\begin{align*}
g \in N \rightarrow j_1(\partial\Omega, g,x) \in { \rm Jet}_1(\partial \Omega, x)
\end{align*}
is injective. 
\end{enumerate}
\end{corollary} 

\begin{remark} \ \begin{enumerate}
\item CR-automorphisms of $\partial \Bb_d$ are determined by their 2-jets, but not their 1-jets. So Corollary~\ref{cor:finite_jet} seems optimal. 
\item It was previously known that if $\Omega \subset \Cb^d$ is a bounded pseudoconvex domain with real analytic boundary, then there exists some $k >0$ such that any biholomorphism is determined by its $k$-jet at a boundary point, see~\cite[Theorem 5]{BER2000}. In the special case that $d=2$ and $\partial \Omega$ is real analytic, it was previously known that $k=2$ is sufficient, see~\cite{ELZ2003}.
\item The proof of Corollary~\ref{cor:finite_jet} part (3) is based on an argument of Huang~\cite{H1993}.
\end{enumerate}
\end{remark}

A theorem of Tits states that a subgroup of $\GL_N(\Rb)$ either contains a free group or has a finite index solvable subgroup. Using Theorem~\ref{thm:main_ft} we will prove the following. 

\begin{corollary} (See Section~\ref{sec:tits_alternative})\label{cor:tits} Suppose $\Omega \subset \Cb^d$ is a bounded pseudoconvex domain with real analytic boundary. If $H \leq \Aut(\Omega)$ is a subgroup, then either $H$ contains a free group or has a finite index solvable subgroup. 
\end{corollary} 

\begin{remark}\label{rmk:real_analytic} \ \begin{enumerate}
\item A result of Diederich and Forn{\ae}ss~\cite{DF1978} implies that every bounded pseudoconvex domain with real analytic boundary has finite type, see the discussion in~\cite[Section 4.1.4]{DA1993}.
\item  In the proof of Corollary~\ref{cor:tits} we consider three cases: when $\Aut(\Omega)$ is compact, when $\Lc(\Omega)$ is a single point, and when $\Lc(\Omega)$ contains at least two distinct points. The assumption that $\partial \Omega$ is real analytic instead of just having finite type is only used in the case when $\Lc(\Omega)$ is a single point. 
\end{enumerate}
\end{remark}

 \subsection{Prior Work and Motivation}Our main motivation for Theorem~\ref{thm:main_ft} comes from the old problem of characterizing, up to biholomorphism, the domains which have large automorphism groups and reasonable boundaries. This can be seen as an analogue of the Riemann Mapping Theorem for higher dimensions. 
 
 The first major result along these lines is the Wong-Rosay Ball Theorem. 
 
\begin{theorem}[Wong and Rosay Ball Theorem \cite{R1979, W1977}]
Suppose $\Omega \subset \Cb^d$ is a bounded strongly pseudoconvex domain. Then $\Aut(\Omega)$ is non-compact if and only if $\Omega$ is biholomorphic to the unit ball. 
\end{theorem}
 
Since Wong and Rosay's work, there have been a number of characterizations of domains with non-compact automorphism group amongst special classes of domains, see for instance~\cite{GK1987, K1992, BP1994, W1995, V2009} and the survey paper~\cite{IK1999}. In this paper we focus on the following related problem: characterize the possible automorphism groups of domains with reasonable boundaries. 

Theorem~\ref{thm:main_ft} is also motivated by a result of Zaitsev who proved for algebraic domains that $\Aut(\Omega)$ has finitely many components. 
 
 \begin{theorem}\label{thm:D}\cite[Theorem 1.2, Corollary 1.1]{Z1995}
 Suppose $D \subset \Cb^d$ is a bounded algebraic domain. Then $\Aut(D)$ has finitely many components. 
 \end{theorem}
 
 \begin{remark} A domain $\Omega \subset \Cb^d$ is called \emph{a bounded algebraic domain} if there exists a real valued polynomial $p: \Cb^d \rightarrow \Rb$ such that $\Omega$ is a bounded connected component of $\{ z \in \Cb^d : p(z) < 0\}$ and $\nabla p(z) \neq 0$ for all $z \in \partial \Omega$. \end{remark}
 
 Zaitsev actually shows that $\Aut(D)$ has the structure of an affine Nash group such that the map $\Aut(D) \times D \rightarrow D$ is Nash. It then follows from basic properties of such groups that $\Aut(D)$ has finitely many components. Our approach to showing the biholomorphism group has finitely many components is different and is based on the classical fact that the outer automorphism group of a semisimple Lie group is finite.
 
 Another motivation for Theorem~\ref{thm:main_ft} comes from work of Isaev and Krantz. Suppose $M$ is a Kobayashi hyperbolic complex manifold, then a biholomorphism of $M$ is determined by its 1-jet of any point. So when $M$ has complex dimension $d$, the automorphism group $\Aut(M)$ has real dimension at most 
 \begin{align*}
 \dim_{\Rb} M + \dim_{\Rb} { \rm U}(d) = 2d+d^2.
 \end{align*}
Further, if $\dim_{\Rb} \Aut(M)=2d+d^2$ it is easy to see that $M$ must be biholomorphic to the unit ball in $\Cb^d$. In fact, there is a gap in the dimension of $\Aut(M)$.
  
 \begin{theorem}[Isaev and Krantz~\cite{IK2001}]\label{thm:gap_auto} Suppose $M$ is a Kobayashi hyperbolic complex manifold. If $M$ has complex dimension $d$ and is not biholomorphic to the unit ball in $\Cb^d$, then 
 \begin{align*}
 \dim_{\Rb} \Aut(M) \leq  2+d^2.
 \end{align*}
 \end{theorem}
 
This gap in the dimension of $\Aut(M)$ motivated our investigation into the possible dimensions of $\Lc(\Omega)$.

To the best of our knowledge, Theorem~\ref{thm:main_ft} is the first result which establishes for a large classes of bounded domains that the automorphism group must be, up to a compact factor and finite index subgroup, a specific Lie group.

\subsection{Structure of the paper} Section~\ref{sec:Prelims} contains some preliminary remarks. Sections~\ref{sec:elem_auto_ft} through~\ref{sec:proof_ft} are devoted to the proof of Theorem~\ref{thm:main_ft}. The proofs of Corollaries~\ref{cor:finite_jet} and~\ref{cor:tits} appear in Sections~\ref{sec:finite_jet} and~\ref{sec:tits_alternative} respectively. At the end of the paper, there is a brief appendix describing some basic properties of semisimple Lie groups and the symmetric spaces they act on. 

  \subsection{Outline of the Proof of Theorem~\ref{thm:main_ft}}

The starting point of our proof is the following result of S.Y. Kim~\cite{K2010}.
 
 \begin{theorem}[S.Y. Kim]\label{thm:SYkim} Suppose $\Omega \subset \Cb^d$ is a bounded pseudoconvex domain with $C^\infty$ boundary and finite type. If there exits an automorphism $\varphi \in \Aut(\Omega)$ such that $\varphi^n(z) \rightarrow x^+$ and $\varphi^{-n}(z) \rightarrow x^-$ for some $z \in \Omega$ and distinct $x^+,x^- \in \partial \Omega$, then $\Omega$ is biholomorphic to a weighted homogeneous polynomial domain. 
 \end{theorem}
 
 \begin{remark} Given $\Omega$ and $\varphi \in\Aut(\Omega)$ as in the statement of Theorem~\ref{thm:SYkim}, work of Bell and Ligocka~\cite{BL1980} and Catlin~\cite{C1987} implies that $\varphi$ extends to a diffeomorphism of $\partial \Omega$. Then it is easy to see that $\varphi(x^\pm) = x^\pm$.  Kim's strategy is to show that $d(\varphi)_{x^+}$ is a hyperbolic matrix and then construct a linearization of the action of $\varphi$ on $\partial \Omega$ near $x^+$.  See~\cite{K2012, KK2008b, S1997, S1995} for similar results.
 \end{remark}
 
In Section~\ref{sec:elem_auto_ft}, we show that when $\Omega$ is a bounded pseudoconvex domain with finite type and $\Lc(\Omega)$ contains two points, then there exists some $\varphi \in \Aut(\Omega)$ such that the forward orbit and backward orbit of $\varphi$ accumulate on two different points of $\partial \Omega$. Hence by S.Y. Kim's result, $\Omega$ is biholomorphic to a weighted homogeneous polynomial domain and in particular $\Aut_0(\Omega)$ is non-trivial.

The next step in the proof is to use a result from~\cite{BZ2017} which shows that the Kobayashi metric on a finite type domain behaves, in some sense, like a negatively curved Riemannian manifold, see Theorem~\ref{thm:visible_ft} below. In Section~\ref{sec:solvable_subgp}, we use this result  to restrict the possible solvable subgroups of $\Aut(\Omega)$. This allows us to deduce, in Section~\ref{sec:proof_ft}, that $\Aut(\Omega)$ has finitely many components and is the almost direct product of a compact subgroup $N$ and a simple Lie group $G$ with real rank one and finite center.

Since $G$ has real rank one, $G$ acts by isometries on a negatively curved Riemannian symmetric space $X$. By the classification of such spaces $X$ is either a real hyperbolic space, a complex hyperbolic space, a quaternionic hyperbolic space, or the Cayley-hyperbolic plane. We will construct a $G$-equivariant diffeomorphism from the geodesic boundary of $X$ to $\Lc(\Omega)$. Then by using the complex geometry of $\Lc(\Omega)$ and facts about negatively curved symmetric spaces, we will deduce that $X$ must be a complex hyperbolic space and hence $G$ is locally isomorphic to $\SU(1,k)$ for some $k$. This implies that $G/Z(G)$ is isomorphic to $\Aut(\Bb_k)$.

 \subsection*{Acknowledgements} 
I would like to thank Gautam Bharali, Nordine Mir, Ralf Spatzier, and Dmitri Zaitsev for helpful comments on an earlier version of this manuscript. This material is based upon work supported by the National Science Foundation under grants DMS-1400919, DMS-1700079, DMS-1760233, and DMS-1904099.

 \section{Preliminaries}\label{sec:Prelims}
 
 \subsection{Finite type domains}\label{sec:FT} In this subsection we recall the definition of finite type (in the sense of D'Angelo). 
 
 Let $M$ be a $C^\infty$-smooth real hypersurface in $\Cb^d$ and let $r$ be a local defining function for $M$. For $p \in \Cb^d$, let $\mathcal{G}(0,p)$ denote the set of germs of non constant holomorphic maps $f$ from $\Cb$ to $\Cb^d$, such that $f(0) = p$. If $g$ is a smooth function defined in a neighborhood of $0 \in \Cb$, we denote by $\nu(g)$ the order of vanishing of the function $g - g(0)$ at the origin. Following~\cite{DA1979}, the {\sl type} $\tau(M,p)$ of $M$ at $p \in M$ is defined by
$$
\tau(M,p):=\sup_{f \in \mathcal{G}(0,p)}\frac{\nu(r \circ f)}{\nu(f)}.
$$
Then the hypersurface $M$ is of {\sl finite type} (in the sense of D'Angelo) if $\tau(M,p) < \infty$ for every $p \in M$. 

Through out the paper we will use the following terminology. 

\begin{definition}\label{defn:FT} A bounded pseudoconvex domain $\Omega \subset \Cb^d$ has \emph{finite type} if $\partial \Omega$ is $C^\infty$ and has finite type (in the sense of D'Angelo).
\end{definition}
 
  \subsection{The Kobayashi metric} 

 In this expository subsection we recall the definition of the Kobayashi metric and state some of its basic properties.  For a more thorough introduction see~\cite{A1989} or~\cite{K2005}.
 
 Given a domain $\Omega \subset \Cb^d$ the \emph{(infinitesimal) Kobayashi metric} is the pseudo-Finsler metric
\begin{align*}
k_{\Omega}(x;v) = \inf \left\{ \abs{\xi} : f \in \Hol(\Db, \Omega), \ f(0) = x, \ d(f)_0(\xi) = v \right\}.
\end{align*}
By a result of Royden~\cite[Proposition 3]{R1971} the Kobayashi metric is an upper semicontinuous function on $\Omega \times \Cb^d$. In particular, if $\sigma:[a,b] \rightarrow \Omega$ is an absolutely continuous curve (as a map $[a,b] \rightarrow \Cb^d$), then the function 
\begin{align*}
t \in [a,b] \rightarrow k_\Omega(\sigma(t); \sigma^\prime(t))
\end{align*}
is integrable and we can define the \emph{length of $\sigma$} to  be
\begin{align*}
\ell_\Omega(\sigma)= \int_a^b k_\Omega(\sigma(t); \sigma^\prime(t)) dt.
\end{align*}
One can then define the \emph{Kobayashi pseudo-distance} to be
\begin{multline*}
 K_\Omega(x,y) = \inf \left\{\ell_\Omega(\sigma) : \sigma\colon[a,b]
 \rightarrow \Omega \text{ is abs. cont., } \sigma(a)=x, \text{ and } \sigma(b)=y\right\}.
\end{multline*}
This definition is equivalent to the standard definition using analytic chains by a result of Venturini~\cite[Theorem 3.1]{V1989}.

When $\Omega$ is a bounded domain, $K_\Omega$ is actually a metric. Further, directly from the definition, $\Aut(\Omega)$ acts by isometries on $(\Omega, K_\Omega)$. 

\subsection{Almost-geodesics}

A \emph{geodesic} in a metric space $(X,d)$ is a curve $\sigma: I \rightarrow X$ such that 
\begin{align*}
d(\sigma(s), \sigma(t))=\abs{t-s}
\end{align*}
for all $s,t \in I$. When the Kobayashi metric is Cauchy complete, every two points are joined by a geodesic. However, it is often more convenient to work with larger classes of curves.

\begin{definition}\label{def:almost_geodesic}
Suppose $\Omega \subset \Cb^d$ is a bounded domain and $I \subset \Rb$ is an interval. For $\lambda \geq 1$ and $\kappa \geq 0$ a curve $\sigma:I \rightarrow \Omega$ is called an \emph{$(\lambda, \kappa)$-almost-geodesic} if 
\begin{enumerate} 
\item for all $s,t \in I$  
\begin{align*}
\frac{1}{\lambda} \abs{t-s} - \kappa \leq K_\Omega(\sigma(s), \sigma(t)) \leq \lambda \abs{t-s} +  \kappa;
\end{align*}
\item $\sigma$ is absolutely continuous (hence $\sigma^\prime(t)$ exists for almost every $t\in I$), and for almost every $t \in I$
\begin{align*}
k_\Omega(\sigma(t); \sigma^\prime(t)) \leq \lambda e^{\kappa}.
\end{align*}
\end{enumerate}
\end{definition}

\begin{remark} In~\cite[Proposition 4.6]{BZ2017}, we proved that every geodesic in the Kobayashi metric is an $(1,0)$-almost-geodesic. \end{remark}

There are several reasons to study almost-geodesics instead of geodesics. First almost-geodesics always exist: for domains $\Omega$ where the metric space $(\Omega, K_\Omega)$ is not Cauchy complete there may not be a geodesic joining every two points, but there is always an $(1,\kappa)$-almost-geodesic joining any two points in $\Omega$, see \cite[Proposition 4.4]{BZ2017}. Further, it is sometimes possible to find explicit almost-geodesics, see for instance Proposition~\ref{prop:translate_ft} below. Finally, almost-geodesics are close enough to geodesics that one can establish properties about their behavior, see Theorem~\ref{thm:visible_ft} below.

 \subsection{The geometry of the Kobayashi metric}
 
 In this subsection we recall some results about the geometry of the Kobayashi metric on finite type domains. It is unknown if the Kobayashi metric is Cauchy complete on every finite type domain, but we still have some negative curvature type behavior. 
 
 \begin{theorem}\cite[Theorem 1.4]{BZ2017}\label{thm:visible_ft} Fix $\lambda \geq 1$ and $\kappa \geq 0$. Suppose $\Omega \subset \Cb^d$ is a bounded pseudoconvex domain with $C^\infty$ boundary and finite type. Assume that $\sigma_n:[a_n,b_n] \rightarrow \Omega$ is a sequence of $(\lambda, \kappa)$-almost-geodesics. If $\sigma_n(a_n) \rightarrow x \in \partial \Omega$, $\sigma_n(b_n) \rightarrow y \in \partial \Omega$, and $x \neq y$, then there exists $n_k \rightarrow \infty$ and $t_k \in [a_{n_k}, b_{n_k}]$ such that the sequence $\sigma_{n_k}(t_k)$ converges in $\Omega$.  \end{theorem}
 
  \begin{remark} Informally this theorem says that almost-geodesics bend into the domain like geodesics in the Poincar{\'e} disc model of real hyperbolic 2-space.
 \end{remark}
 
 As a corollary we have the following.
 
 \begin{corollary}  Suppose $\Omega \subset \Cb^d$ is a bounded psuedoconvex domain with $C^\infty$ boundary and finite type. If $\sigma:[0,\infty) \rightarrow \Omega$ is an almost-geodesic, then
 \begin{align*}
 \lim_{t \rightarrow \infty} \sigma(t)
 \end{align*}
 exists.
 \end{corollary} 
 
 \begin{proof}
 Suppose that $\sigma:[0,\infty) \rightarrow \Omega$ is an almost-geodesic and there exists sequences $s_n, t_n \rightarrow \infty$ such that $\sigma(s_n) \rightarrow x$, $\sigma(t_n) \rightarrow y$, and $x \neq y$. By passing to subsequences we can further assume that $s_n \leq t_n$ for all $n$. Since $x \neq y$, Theorem~\ref{thm:visible_ft} implies that
  \begin{align*}
 \sup_{n \in \Nb} K_\Omega\Big( \sigma(0), \sigma([s_n,t_n]) \Big) < \infty.
 \end{align*}
But since $\sigma$ is a almost-geodesic we have
   \begin{align*}
K_\Omega\Big( \sigma(0), \sigma([s_n,t_n]) \Big)\geq \frac{1}{\lambda}s_n - \kappa
 \end{align*}
for some $\lambda \geq 1$ and $\kappa \geq 0$. So we have a contradiction.
 \end{proof}

 \subsection{Smooth extensions to the boundary} By results of Catlin~\cite{C1987} and Bell and Ligocka \cite{BL1980}, if $\Omega$ is a bounded pseudoconvex domain with $C^\infty$ boundary and finite type, then each $\varphi \in \Aut(\Omega)$ extends to a diffeomorphism of $\overline{\Omega}$. Later Bell proved that the induced homomorphism $\Aut(\Omega) \rightarrow { \rm Diff}(\overline{\Omega})$ is continuous in the Whitney topology, see~\cite{B1987}. This implies, by a classical result of Montgomery and Zippin, that the map 
  \begin{align*}
 \Aut(\Omega) \times \overline{\Omega} \rightarrow \overline{\Omega} \\
 (\varphi,z) \rightarrow \varphi(z)
 \end{align*}
 is smooth, see~\cite[Chapter 5]{MZ1956}.  So:
  
 \begin{theorem}\label{thm:Bell1} Suppose $\Omega \subset \Cb^d$ is a bounded pseudoconvex domain with $C^\infty$ boundary and finite type. The map 
 \begin{align*}
 \Aut(\Omega) \times \Omega \rightarrow \Omega \\
 (\varphi,z) \rightarrow \varphi(z)
 \end{align*}
 extends to a smooth map $\Aut(\Omega) \times \overline{\Omega} \rightarrow \overline{\Omega}$. 
 \end{theorem}

 We will also use the following theorem of Bell.
 
 \begin{theorem}[Bell~\cite{B1987}]\label{thm:Bell2} Suppose $\Omega \subset \Cb^d$ is a bounded pseudoconvex domain with $C^\infty$ boundary and finite type. If $z_0 \in \Omega$ and $\varphi_n  \in \Aut(\Omega)$ is a sequence of automorphisms with $\varphi_n(z_0) \rightarrow x \in \partial \Omega$ and $\varphi_n^{-1}(z_0) \rightarrow y \in \partial \Omega$. Then $\varphi_n(z)$ converges locally uniformly on $\overline{\Omega} \setminus \{ y\}$ to $x$ and $\varphi_n^{-1}(z)$ converges locally uniformly on $\overline{\Omega} \setminus \{ x\}$ to $y$.
 \end{theorem}
 
 \subsection{Limit sets of subgroups}
 
Given a domain $\Omega \subset \Cb^d$ and a subgroup $H \leq \Aut(\Omega)$ the \emph{limit set of $H$}, denoted $\Lc(\Omega;H)$, is the set of points $x \in \partial \Omega$ where there exists some $z \in \Omega$ and a sequence $h_n \in H$ such that $h_n(z) \rightarrow x$.
 
 \begin{proposition}\label{prop:subgroups} Suppose $\Omega \subset \Cb^d$ is a bounded pseudoconvex domain with $C^\infty$ boundary and finite type. If $H \leq \Aut(\Omega)$ is a subgroup, then $\Lc(\Omega; H)$ is a closed subset of $\partial \Omega$. If $N \leq \Aut(\Omega)$ normalizes $H$, then $\Lc(\Omega; H)$ is $N$-invariant. 
 \end{proposition}
 
 \begin{proof} Suppose $x_m \in \Lc(\Omega; H)$ and $x_m \rightarrow x \in \partial \Omega$. Then there exists $z_m \in \Omega$ and sequences $\varphi^{(m)}_n \in H$ such that $\lim_{n \rightarrow \infty} \varphi^{(m)}_n(z_m) = x_m$. Then by Theorem~\ref{thm:Bell2} 
 \begin{align*}
 \lim_{n \rightarrow \infty} \varphi^{(m)}_n(z) = x_m
 \end{align*}
 for any $z \in \Omega$. So we can find $n_m \rightarrow \infty$ such that 
 \begin{align*}
 \lim_{m \rightarrow \infty} \varphi^{(m)}_{n_m}(z) = x.
 \end{align*}
 Thus  $x \in \Lc(\Omega;H)$ and hence $\Lc(\Omega; H)$ is a closed subset of $\partial \Omega$
 
 Now suppose that  $N \leq \Aut(\Omega)$ normalizes $H$. If $x \in \Lc(\Omega; H)$, then there exists some  $z \in \Omega$ and a sequence $\varphi_m \in H$ such that $\lim_{m \rightarrow \infty} \varphi_m(z) = x$. Now if $n \in N$, then Theorems~\ref{thm:Bell1} and~\ref{thm:Bell2} imply that 
 \begin{align*}
 \lim_{m \rightarrow \infty} n\varphi_mn^{-1}(z) =n\left( \lim_{m \rightarrow \infty} \varphi_m(n^{-1}(z)) \right)=n(x).
 \end{align*}
 So $n x \in \Lc(\Omega; H)$ and hence $\Lc(\Omega; H)$ is $N$-invariant.
 \end{proof}

\section{Elements of the automorphism group}\label{sec:elem_auto_ft}

For bounded domains with finite type boundary we have the following analogue of the Wolff-Denjoy theorem. 

\begin{theorem}\label{thm:wolf_ft}\cite[Corollary 2.11]{BZ2017} Suppose $\Omega \subset \Cb^d$ is a bounded pseudoconvex domain with $C^\infty$ boundary and finite type. If $f: \Omega \rightarrow \Omega$ is a holomorphic map, then either 
\begin{enumerate}
\item for every $z \in \Omega$ the orbit $\{f^n(z) : n \in \Nb\}$ is relatively compact in $\Omega$, 
\item there exists a point $\ell \in \partial \Omega$ such that 
\begin{align*}
\lim_{n \rightarrow \infty} f^n(z) = \ell
\end{align*}
for all $z \in \Omega$. 
\end{enumerate}
\end{theorem}

\begin{remark} Karlsson~\cite{K2005b} proved Theorem~\ref{thm:wolf_ft} with the additional assumption that the metric space $(\Omega, K_\Omega)$ is Cauchy complete. \end{remark}

Using Theorem~\ref{thm:wolf_ft} we can characterize the automorphisms of $\Omega$ by the behavior of their iterates. Suppose $\Omega \subset \Cb^d$ is a bounded pseudoconvex domain with $C^\infty$ boundary and finite type. If $\varphi \in \Aut(\Omega)$, then by Theorem~\ref{thm:wolf_ft} either $\varphi$ has relatively compact orbits $\Omega$ or there exists some point $\ell_\varphi^+ \in \partial \Omega$ such that 
\begin{align*}
\lim_{n \rightarrow \infty} \varphi^n(z) =\ell_\varphi^+
\end{align*}
 for all $z \in \Omega$. In this latter case, we call $\ell_\varphi^+$ the \emph{attracting fixed point of} $\varphi$. 
 
  \begin{definition}\label{defn:elems_ft}
Suppose $\Omega \subset \Cb^d$ is a bounded pseudoconvex domain with $C^\infty$ boundary and finite type. If $\varphi \in \Aut(\Omega)$, then:
\begin{enumerate}
\item $\varphi$ is \emph{elliptic} if every orbit of $\varphi$ in $\Omega$ is relatively compact in $\Omega$, 
\item $\varphi$ is \emph{parabolic} if $\varphi$ is not elliptic and $\ell_\varphi^+ = \ell_{\varphi^{-1}}^+$,
\item $\varphi$ is \emph{hyperbolic} if $\varphi$ is not elliptic and $\ell_\varphi^+ \neq \ell_{\varphi^{-1}}^+$. In this case we call $ \ell_{\varphi}^-:= \ell_{\varphi^{-1}}^+$  the \emph{repelling fixed point of} $\varphi$.
\end{enumerate}
\end{definition}

\begin{remark} Theorem~\ref{thm:wolf_ft} implies that every automorphism of $\Omega$ is either elliptic, hyperbolic, or parabolic. 
\end{remark}

\subsection{The dynamics} 

We have the following immediate consequences of Theorem~\ref{thm:Bell2} and the definitions.

\begin{corollary}\label{cor:NS} Suppose $\Omega \subset \Cb^d$ is a bounded pseudoconvex domain with $C^\infty$ boundary and finite type. Assume $h \in \Aut(\Omega)$ is hyperbolic.  If $U$ is a neighborhood of $\ell_h^+$ in $\overline{\Omega}$ and $V$ is a neighborhood of $\ell_h^-$ in $\overline{\Omega}$, then there exists some $N >0$ such that 
 \begin{align*}
h^n \left(\overline{\Omega} \setminus V\right) \subset U \text{ and } h^{-n}\left(\overline{\Omega} \setminus U\right) \subset V
 \end{align*}
 for all $n \geq N$.
\end{corollary}

\begin{corollary}\label{cor:parabolic_dynamics} Suppose $\Omega \subset \Cb^d$ is a bounded pseudoconvex domain with $C^\infty$ boundary and finite type. Assume $u \in \Aut(\Omega)$ is parabolic.  If $U$ is a neighborhood of $\ell_u^+$ in $\overline{\Omega}$, then there exists some $N >0$ such that 
 \begin{align*}
u^n \left(\overline{\Omega} \setminus U\right) \subset U \text{ and } u^{-n}\left(\overline{\Omega} \setminus U\right) \subset U \end{align*}
 for all $n \geq N$.
\end{corollary}

\subsection{Constructing hyperbolic elements}

\begin{lemma}\label{lem:construct_hyp_ft}  Suppose $\Omega \subset \Cb^d$ is a bounded pseudoconvex domain with $C^\infty$ boundary and finite type. Assume $\phi_n \in \Aut(\Omega)$ is a sequence of automorphisms with $\phi_n(z) \rightarrow x^+$ and $\phi_n^{-1}(z) \rightarrow x^-$ for some $z \in \Omega$ and $x^+,x^- \in \partial \Omega$. If $x^+ \neq x^-$, then $\phi_n$ is hyperbolic for $n$ large. Further, $\ell_{\phi_n}^\pm \rightarrow x^\pm$. 
\end{lemma}

\begin{proof} 
Fix disjoint neighborhoods $U^+,U^-$ of $x^+,x^-$ in $\overline{\Omega}$. By Theorem~\ref{thm:Bell2} there exists some $N \geq 0$ such that 
\begin{align*}
\phi_n\left( \Omega \setminus U^- \right) \subset U^+ \text{ and } \phi_n^{-1}\left( \Omega \setminus U^+ \right) \subset U^-
\end{align*}
for all $n \geq N$. So 
\begin{align*}
\phi_n^m\left( \Omega \setminus U^- \right) \subset U^+ \text{ and } \phi_n^{-m}\left( \Omega \setminus U^+\right) \subset U^-
\end{align*}
for all $n \geq N$ and $m \in \Nb$. So from Corollary~\ref{cor:parabolic_dynamics}, we see that $\phi_n$ is not parabolic for $n \geq N$. Further, if $\phi_n$ is elliptic for some $n \in \Nb$, then $\{ \phi_n^m : m \in \Zb\}$ is relatively compact in $\Aut(\Omega)$. So there exists some $m_k \rightarrow \infty$ such that 
\begin{align*}
\lim_{k \rightarrow \infty} \phi_n^{m_k}(z) =z 
\end{align*}
for all $z \in \Omega$ which is impossible when $n \geq N$. So we see that $\phi_n$ is not elliptic when $n \geq N$. So $\phi_n$ must be hyperbolic for $n \geq N$.

We next show that $\ell_{\phi_n}^+ \rightarrow x^+$. Since 
\begin{align*}
\phi_n^m\left( \Omega \setminus U^- \right) \subset U^+
\end{align*}
for all $n \geq N$ and $m \in \Nb$, Corollary~\ref{cor:NS} implies that $\ell_{\phi_n}^+ \in U^+$ when $n \geq N$. Since $U^+$ was an arbitrary neighborhood of $x^+$ we then see that $\ell_{\phi_n}^+ \rightarrow x^+$.

To show that  $\ell_{\phi_n}^- \rightarrow x^-$ one simply repeats the argument above. 
\end{proof}

\begin{proposition}\label{prop:construct_hyp} Suppose $\Omega \subset \Cb^d$ is a bounded pseudoconvex domain with $C^\infty$ boundary and finite type. If $H \leq \Aut(\Omega)$ is a subgroup and $\Lc(\Omega; H)$ contains at least two points, then $H$ contains a hyperbolic element. \end{proposition}

\begin{proof} Suppose that $\Lc(\Omega; H)$ contains two distinct points $x,y$. Then there exists $\phi_m, \varphi_n \in H$ and $z_1,z_2 \in \Omega$ such that $\phi_m(z_1) \rightarrow x$ and $\varphi_n(z_2) \rightarrow y$. By passing to a subsequence we can suppose that $\phi_m^{-1}(z_1) \rightarrow x^-$ and $\varphi_m^{-1}(z_2) \rightarrow y^-$. Now if $x\neq x^-$, then Lemma~\ref{lem:construct_hyp_ft} implies that $\phi_m$ is hyperbolic for large $m$ and there is nothing to prove. Likewise, if $y \neq y^-$, then $\varphi_n$ is hyperbolic for large $n$. So we may assume that $x=x^-$ and $y=y^-$. 

Then by Theorem~\ref{thm:Bell2}, we see that $\phi_n(z)$ converges locally uniformly to $x$ on $\overline{\Omega}\setminus \{x\}$ and $\varphi_m^{-1}(z)$ converges locally uniformly to $y$ on $\overline{\Omega} \setminus \{y\}$. Since $x \neq y$, if  $h_n = \phi_n \varphi_n$ then $h_n(z) \rightarrow x$ and $h_n^{-1}(z) \rightarrow y$ for all $z \in \Omega$. So Lemma~\ref{lem:construct_hyp_ft} implies that $h_n$ is hyperbolic for large $n$. 

\end{proof}

\subsection{Ping-Pong}

The next proposition is not used in the proof of Theorem~\ref{thm:main_ft}, but naturally fits into the current discussion. 

\begin{proposition}\label{prop:PP1_ft} Suppose $\Omega \subset \Cb^d$ is a bounded pseudoconvex domain with $C^\infty$ boundary and finite type. If $h_1, h_2 \in \Aut(\Omega)$ are hyperbolic elements and 
\begin{align*}
\ell_{h_1}^+, \ell_{h_1}^-, \ell_{h_2}^+, \ell_{h_2}^-
\end{align*}
are all distinct, then there exists $n,m >0$ such that the elements $h_1^m, h_2^n$ generate a free group.
\end{proposition}

\begin{proof}
This follows from Corollary~\ref{cor:NS} and the well known ``ping-pong lemma,'' see for instance~\cite[Section II.B]{P2000}.
\end{proof}

\subsection{Hyperbolic elements translate an almost geodesic}

In a ${ \rm CAT}(0)$ metric space a hyperbolic isometry always translates a geodesic (see for instance~\cite[Chapter II.6 Theorem 6.8]{BH1999}). We now show that a similar phenomena holds for hyperbolic automorphisms. 

\begin{proposition}\label{prop:translate_ft} Suppose $\Omega$ is a bounded pseudoconvex domain with $C^\infty$ boundary and finite type. If $h \in \Aut(\Omega)$ is hyperbolic, then there exists some $\lambda \geq 1$, $\kappa \geq 0$, $T > 0$, and an $(\lambda,\kappa)$-almost-geodesic $\gamma: \Rb \rightarrow \Omega$ such that 
\begin{align*}
h^n \gamma(t) = \gamma(t+nT)
\end{align*}
 for all $t \in \Rb$ and $n \in \Zb$. 
\end{proposition}

We start the proof of the proposition with a lemma. 

\begin{lemma}\label{lem:trans_dist_ft} Suppose $\Omega$ is a bounded pseudoconvex domain with $C^\infty$ boundary and finite type. If $h \in \Aut(\Omega)$ is a hyperbolic element, then there exists some $L > 0$ such that 
\begin{align*}
\lim_{n \rightarrow \infty} \frac{K_\Omega(h^n(z), z) }{n} =L
\end{align*}
for all $z \in \Omega$. 
\end{lemma}

\begin{proof} If we fix $z \in \Omega$ and let $b_n = K_\Omega(h^n(z), z)$, then $b_{m+n} \leq b_m + b_n$. So by a standard lemma (see for instance~\cite[Theorem 4.9]{W1982}) the limit 
\begin{align*}
L=\lim_{n \rightarrow \infty} \frac{K_\Omega(h^n(z), z) }{n}
\end{align*}
exists. Further the limit
\begin{align*}
L=\lim_{n \rightarrow \infty} \frac{K_\Omega(h^n(z), z) }{n}
\end{align*}
clearly does not depend on the choice of $z$. So we only need to show that the limit is positive.

For $0 < \epsilon_1 < \epsilon_2$ sufficiently small define 
\begin{align*}
C_0 = \{ z \in \overline{\Omega} : \epsilon_1 \leq \norm{ z-\ell_h^+} \leq \epsilon_2 \}.
\end{align*}
By picking $\epsilon_1, \epsilon_2$ small enough we can assume that $\overline{\Omega} \setminus C_0$ has two connected components $A_0$, $B_0$ with $\ell^+_h \in A_0$ and $\ell_h^- \in B_0$. Now by~\cite[Proposition 3.5]{BZ2017} there exists some $\delta_0 > 0$ such that $k_\Omega(z;v) \geq \delta_0 \norm{v}$ for all $x \in \Omega$ and $v \in \Cb^d$. Define $\delta: = \delta_0(\epsilon_2-\epsilon_1)$. Then, if $z_1 \in A_0 \cap \Omega$ and $z_2 \in B_0 \cap \Omega$ we have $K_\Omega(z_1, z_2) \geq \delta$. 

Now fix some $z_0 \in B_0 \cap \Omega$. By Theorem~\ref{thm:Bell2} there exists some $n_0 > 0$ such that 
\begin{align*}
h^{n_0}\left(A_0 \cup C_0 \cup \{ z_0\} \right) \subset A_0.
\end{align*}
Then for $j \in \Nb$ define 
\begin{align*}
A_j = h^{n_0j}(A_0), \ B_j = h^{n_0j}(B_0), \text{ and } C_j = h^{n_0j}(C_0).
\end{align*}
Then, by construction, if $z_1 \in A_j \cap \Omega$ and $z_2 \in B_j \cap \Omega$ we have $K_\Omega(z_1, z_2) \geq \delta$. 

Now suppose that $\sigma:[0,T] \rightarrow \Omega$ is an absolutely continuous curve from $z_0$ to $h^{n_0N}(z_0)$. Then by construction there exists $0=t_0< t_1 < t_2 < \dots < t_{N}=T$ such that 
\begin{align*}
\sigma(t_j) \in A_{j-1} \cap B_{j} \text{ for } 1 \leq j \leq N-1.
\end{align*}
So 
\begin{align*}
\ell_\Omega(\sigma) 
&= \int_0^T k_\Omega( \sigma(t); \sigma^\prime(t)) dt = \sum_{j=0}^{N-1} \int_{t_j}^{t_{j+1}} k_\Omega( \sigma(t); \sigma^\prime(t)) dt \\
& \geq \sum_{j=0}^{N-1} K_\Omega(\sigma(t_j), \sigma(t_{j+1})) \geq N\delta.
\end{align*}
Since $\sigma$ was an arbitrary absolutely continuous curve from $z_0$ to $h^{n_0N}(z_0)$,  we have 
\begin{align*}
K_\Omega( h^{n_0N}(z_0), z_0) \geq N \delta.
\end{align*}
So $L > \delta/n_0$. 

\end{proof}

\begin{proof}[Proof of Proposition~\ref{prop:translate_ft}] Fix some $z_0 \in \Omega$ and $\kappa_0 > 1$. Then let $\gamma_0:[0,T] \rightarrow \Omega$ be an $(1,\kappa_0)$-almost-geodesic joining $z_0$ to $hz_0$ (such curves exist by~\cite[Proposition 4.4]{BZ2017}). Then define a curve $\gamma : \Rb \rightarrow \Omega$ by letting 
\begin{align*}
\gamma(t) = h^m \gamma_0(t - Tm)
\end{align*}
 when $t \in [mT, (m+1)T]$ and $m \in \Zb$. Clearly $h^n \gamma(t) = \gamma(t+nT)$ for all $t \in \Rb$ and $n \in \Zb$. Further because $\gamma_0$ is a $(1,\kappa_0)$-almost-geodesic we see that $\gamma$ is absolutely continuous and 
\begin{align*}
k_\Omega( \gamma(t); \gamma^\prime(t)) \leq e^{\kappa_0}
\end{align*}
almost everywhere. Then
\begin{align*}
K_\Omega( \gamma(s), \gamma(t)) \leq \abs{ \int_s^t k_\Omega( \gamma(r); \gamma^\prime(r)) dr}  \leq e^{\kappa_0}\abs{t-s}
\end{align*}
for all $s,t \in \Rb$. By the previous lemma there exists some $\alpha,\beta > 0$ such that 
\begin{align*}
K_\Omega( \gamma(mT), \gamma(nT)) = K_\Omega( h^m(z_0), h^{n}(z_0)) = K_\Omega( h^{\abs{m-n}}(z_0), z_0) \geq \alpha \abs{m-n}-\beta
\end{align*}
for all $m, n \in \Zb$. Now if $s,t \in \Rb$ there exists $m,n \in \Zb$ such that $\abs{s-mT} \leq T/2$ and $\abs{t-nT} \leq T/2$. So 
\begin{align*}
K_\Omega( \gamma(s), \gamma(t)) & \geq K_\Omega( \gamma(mT), \gamma(nT)) - K_\Omega(\gamma(mT), s)-K_\Omega(t, \gamma(nT)) \\
&  \geq \alpha \abs{m-n} - \beta - e^{\kappa_0}T \\
& \geq \frac{\alpha}{T} \abs{t-s}-\alpha T-\beta-e^{\kappa_0}T.
\end{align*}
So $\gamma$ is an $(\lambda, \kappa)$-almost-geodesic for some $\lambda > 1$ and $\kappa > 0$. 
\end{proof}

\subsection{More on weighted homogeneous polynomial domains}

In this section we describe some consequences of S.Y. Kim's rigidity result in~\cite{K2010}.

\begin{theorem}\label{thm:SYKim2} Suppose $\Omega$ is a bounded pseudoconvex domain with $C^\infty$ boundary and finite type. If $\Aut(\Omega)$ contains a hyperbolic element, then:
\begin{enumerate}
\item $\Omega$ is biholomorphic to a weighted homogeneous polynomial domain. 
\item If $h \in \Aut(\Omega)$ is a hyperbolic element, then there exists a one-parameter group $u_t \in \Aut(\Omega)$ such that 
\begin{align*}
\left.\frac{d}{dt}\right|_{t=0} u_t(\ell^+_h) \notin T_{\ell^+_h}^{\Cb} \partial \Omega
\end{align*}
and $u_t(\ell^-_h) = \ell^-_h$. 
\item There exists a hyperbolic element in $\Aut_0(\Omega)$. 
\item If $x_1, \dots, x_N \in \partial \Omega$, then there exists a hyperbolic element $\phi \in \Aut_0(\Omega)$ such that 
\begin{align*}
\ell^+_\phi,\ell^-_\phi \notin \{ x_1, \dots, x_N\}.
\end{align*}
\item $\Aut_0(\Omega)$ acts without fixed points on $\partial \Omega$, that is if $z \in \partial \Omega$, then there exists $g \in \Aut_0(\Omega)$ with $g(z) \neq z$
\end{enumerate}
\end{theorem}

\begin{proof} Part (1) is just Corollary 1 in~\cite{K2010}. 

Part (2) and Part (3) follow from the proof of Theorem 2 in~\cite[Section 6]{K2010}. In particular, if $h \in \Aut(\Omega)$ is hyperbolic the discussion on page 432 in~\cite{K2010} implies that there exists an weighted homogeneous polynomial domain
\begin{align*}
\Pc = \{ (w, z) \in \Cb \times \Cb^{d-1} : {\rm Im}(z) > p(z) \}
\end{align*}
and a biholomorphism $\Psi: \Omega \rightarrow \Pc$ with the following properties:
\begin{enumerate}
\item if $\wt{h} = \Psi \circ h \circ \Psi^{-1}$, then 
\begin{align*}
\wt{h}(w, z) = (\mu w, D z)
\end{align*}
for some $0 < \mu < 1$ and $D$ a diagonal complex matrix,
\item there exists a neighborhood $U$ of $\ell^+_h$ in $\partial \Omega$ where $\Psi$ extends to a smooth map $\overline{\Psi}:U \rightarrow \partial \Pc$ and $\overline{\Psi}(\ell^+_h)=0$.
\item $\overline{\Psi}$ is an infinitesimal CR-automorphism (see page 431 in~\cite{K2010}).
\end{enumerate}

Now let $\wt{u}_t: \Pc \rightarrow \Pc$ be the one-parameter group of automorphisms $\wt{u}_t(w,z) = (w+t,z)$ and let $u_t = \Psi^{-1} \circ \wt{u}_t \circ \Psi$. Using the fact that $\overline{\Psi}$ is an infinitesimal CR-automorphism we see that 
\begin{align*}
\left.\frac{d}{dt}\right|_{t=0} u_t(\ell^+_h) \notin T_{\ell^+_h}^{\Cb} \partial \Omega.
\end{align*}
Further, 
\begin{align*}
\wt{h}^n \wt{u}_t \wt{h}^{-n}(w,z) = \left(w+\mu^n t,z\right) = \wt{u}_{\mu^n t}
\end{align*}
so 
\begin{align*}
\lim_{n \rightarrow \infty} h^n u_t h^{-n} = \id
\end{align*}
in $\Aut(\Omega)$. Next fix some $z_0 \in \Omega$. Then
\begin{align*}
u_t(\ell^-_h) = u_t \left( \lim_{n \rightarrow \infty} h^{-n} z_0 \right)  = \lim_{n \rightarrow \infty} h^{-n} \left( h^n u_t  h^{-n} z_0\right) = \ell^-_h.
\end{align*}
This establishes Part (2).

We now prove Part (3). Since $p$ is a weighted homogeneous polynomial, there exist a one-parameter group of the form 
\begin{align*}
\wt{a}_t(w,z)  = (e^{t} w, A_t z)
\end{align*}
where $A_t$ is a matrix. Then let $a_t = \Psi^{-1} \circ \wt{a}_t \circ \Psi$.  Since
\begin{align*}
a_{n\log(\mu)}(\Psi^{-1}(w,0)) = h^n(\Psi^{-1}(w,0)),
\end{align*}
we see that $\ell^\pm_{a_t} = \ell^\pm_h$ when $t > 0$. So $a_t$ is hyperbolic when $t 
\neq 0$. This establishes Part (3). 

We now prove Part (4). By Part (3), there exists an hyperbolic element $g \in \Aut_0(\Omega)$. Then by Part (2), there exist two one-parameter subgroups $u_t^+, u_t^-$ such that  
 \begin{align*}
u_t^+(\ell^-_{g}) = \ell^-_{g} \text{ and } u_t^-(\ell^+_{g}) = \ell^+_{g}.
 \end{align*}
Further,  the maps
  \begin{align*}
 t \rightarrow u_t^+(\ell_{g}^+) \text{ and }  t \rightarrow u_t^-(\ell_{g}^-)
 \end{align*}
 are non-constant. By Theorem~\ref{thm:Bell1} the curves
\begin{align*}
t \rightarrow u_t^+(\ell_g^+) \text{ and } t \rightarrow u_t^-(\ell_g^-)
\end{align*}
are continuous, so we can pick $t_1,t_2$ such that 
\begin{align*}
u_{t_1}^+(\ell_{g}^+), u_{t_2}^-(\ell_{g}^-) \notin \{ x_1, \dots, x_N\}
 \end{align*}
 and 
 \begin{align*}
 u_{t_1}^+(\ell_{g}^+), u_{t_2}^-(\ell_{g}^-), \ell_h^+, \ell_h^-
 \end{align*}
 are all distinct. 
 
Now let 
\begin{align*}
g_1 = u_{t_1}^+g u_{-t_1}^+ \text{ and }  g_2 = u_{t_2}^-g u_{-t_2}^-
\end{align*}
then $\ell^+_{g_1}= u_{t_1}^+(\ell_g^+)$,  $\ell^-_{g_1}= u_{t_2}^+(\ell_{g}^-)=\ell_g^-$, $\ell^+_{g_2}= u_{t_2}^-(\ell_g^+)$, and $\ell^-_{g_2}= u_{t_2}^-(\ell_g^+)=\ell_g^+$. So
\begin{align*}
\ell^+_{g_1}, \ell^-_{g_1}, \ell^+_{g_2}, \ell^-_{g_2}
\end{align*}
are all distinct. Then let $\phi_n = g_1^n g_2^{-n}$. By applying Corollary~\ref{cor:NS} to $g_1$ and $g_2$ we see that 
\begin{align*}
\lim_{n \rightarrow \infty} \phi_n z = \ell^+_{g_1}
\end{align*}
 and 
 \begin{align*}
\lim_{n \rightarrow \infty} \phi_n^{-1} z = \ell^+_{g_2}
\end{align*}
for all $z \in \Omega$. Then Lemma~\ref{lem:construct_hyp_ft} implies that $\phi_n$ is hyperbolic for large $n$ with $\ell^+_{\phi_n} \rightarrow \ell^+_{g_1}$ and $\ell^-_{\phi_n} \rightarrow \ell^+_{g_2}$. So for $n$ large enough 
\begin{align*}
\ell_{\phi_n}^+, \ell_{\phi_n}^- \notin  \{ x_1, \dots, x_N\}.
 \end{align*}
 Further, $\phi_n \in \Aut_0(\Omega)$ since $g \in \Aut_0(\Omega)$.
 
 Finally, Part (5) follows from Part (4) and Theorem~\ref{thm:Bell2}. 

\end{proof}

\section{Solvable subgroups}\label{sec:solvable_subgp}

 In this section we establish an analogue of a result of Byers~\cite{B1976}: if $X$ is a complete simply connected Riemannian manifold with sectional curvature bounded above by a negative number and $S$ is a solvable subgroup of the isometry group of $X$, then either $S$ has a fixed point in $X$, a fixed point in the geodesic boundary of $X$, or leaves some geodesic in $X$ invariant.

For finite type domains we prove the following analogue of Byer's theorem.

\begin{theorem}\label{thm:solvable_subgroups} Suppose $\Omega$ is a bounded pseudoconvex domain with $C^\infty$ boundary and finite type. If $S \leq \Aut(\Omega)$ is a closed non-compact solvable subgroup, then either
\begin{enumerate}
\item there exists a term $S_{m+1}$ of the derived series of $S$ such that every element of $S_{m+1}$ is elliptic, $S_{m+1}$ is non-compact, and $\Lc(\Omega;S_{m+1})$ is a single point;
\item $S$ contains a hyperbolic element $h$ such that $S$ preserves the set $\{\ell^+_h, \ell^-_h\}$ and the quotient $S/\{ h^n : n \in \Zb\}$ is compact; or 
\item $S$ contains a parabolic element $u$ and $S$ fixes $\ell^+_u$.
\end{enumerate}
Further, if $N$ is a connected subgroup of $\Aut(\Omega)$ which normalizes $S$, then $N$ has a fixed point in $\partial \Omega$. 
\end{theorem}

\begin{remark} It seems possible that case (1) never actually occurs. In particular, every non-compact solvable subgroup $S$ of $\Aut(\Bb_d)$ contains a hyperbolic or parabolic element, so for $\Bb_d$ case (1) never occurs. More generally, when $\Omega$ is a bounded pseudoconvex domain with finite type and $\Lc(\Omega)$ contains at least two points, then Theorem~\ref{thm:main_ft} implies that case (1) never occurs. 
\end{remark}

\begin{proof} Let $S=S_0 \geq S_1 \geq S_2 \geq \dots \geq S_M = 1$ be the derived series of $S$. Then each $S_j$ is a closed subgroup of $\Aut(\Omega)$. 

Let $m$ be the largest number such that $S_m$ contains a non-elliptic element. In the case in which every element of $S$ is elliptic, let $m=-1$. \newline

\noindent \textbf{Case A:} $S_{m+1}$ is non-compact. Since $S_{m+1}$ is closed and $\Aut(\Omega)$ acts properly on $\Omega$, the limit set $\Lc(\Omega; S_{m+1})$ is non-empty. By assumption, every element of $S_{m+1}$ is elliptic, so Lemma~\ref{lem:construct_hyp_ft} implies that $\# \Lc(\Omega;S_{m+1}) < 2$. So $\Lc(\Omega; S_{m+1})=\{x_0\}$ for some point. Thus we are in case (1) of the theorem. \newline

\noindent \textbf{Case B:} $S_{m+1}$ is compact and $S_m$ contains a hyperbolic element $h$. We first claim that $S_{m+1}$ fixes $\ell_{h}^+$ and $\ell_{h}^-$. Fix some $z_0 \in \Omega$. Then, since $S_{m+1}$ is compact, the set 
\begin{align*}
\{ s z_0 : s \in S_{m+1}\}
\end{align*}
is compact in $\Omega$. Then for $s \in S_{m+1}$ we have
\begin{align*}
s(\ell_h^\pm) = \lim_{n \rightarrow \infty} sh^nz_0 = \lim_{n \rightarrow \infty} h^n (h^{-n} sh^n) z_0 = \ell_h^\pm
\end{align*}
by Theorem~\ref{thm:Bell2} since $h^{-n}sh^n$ is in $S_{m+1}$.

Next we claim that $\{s \ell^+_{h},s \ell^-_{h}\} = \{ \ell^+_{h}, \ell^-_{h}\}$ for every $s \in S$. Suppose $s \in S_i$ for some $i \leq m$ then $shs^{-1}h^{-1} \in S_{i+1}$ so by induction 
\begin{align*}
shs^{-1}h^{-1}\{ \ell^+_{h}, \ell^-_{h}\} = \{ \ell^+_{h}, \ell^-_{h}\}.
\end{align*}
But then
\begin{align*}
shs^{-1}\{ \ell^+_{h}, \ell^-_{h}\} = \{ \ell^+_{h}, \ell^-_{h}\}.
\end{align*}
However, $shs^{-1}$ is hyperbolic with fixed points $s\ell^{\pm}_{h}$. So by Corollary~\ref{cor:NS} we must have that 
\begin{align*}
\{ s\ell^+_{h}, s\ell^-_{h}\} = \{ \ell^+_{h}, \ell^-_{h}\}.
\end{align*}

We now argue that the quotient $S/\{ h^n : n \in \Zb\}$ is compact. So suppose that $s_n \in S$ is a sequence. We claim that there exists $n_k \rightarrow \infty$ and a sequence $m_k \in\Zb$ such that $s_{n_k}h^{m_k}$ converges.  By Proposition~\ref{prop:translate_ft} there exists $\lambda \geq1$, $\kappa \geq 0$, $T >0$, and an $(\lambda, \kappa)$-almost-geodesic $\gamma : \Rb \rightarrow \Omega$ such that 
\begin{align*}
h^m\gamma(t) = \gamma(t+mT)
\end{align*}
for all $t \in \Rb$ and $m \in \Nb$. Since the set $\gamma([0,T]) \subset \Omega$ is compact, Theorem~\ref{thm:Bell2} implies that 
\begin{align*}
\lim_{t \rightarrow \pm \infty} \gamma(t)=\ell^\pm_h.
\end{align*}

Next consider the almost-geodesics $\gamma_n = s_n \gamma$. Since $S\{ \ell^+_{h}, \ell^-_{h}\} = \{ \ell^+_{h}, \ell^-_{h}\}$ we see that
\begin{align*}
\lim_{t \rightarrow \pm \infty} \gamma_n(t) = s_n\left( \lim_{t \rightarrow \pm \infty} \gamma(t)\right)  \in \{\ell^+_h, \ell^-_h\}.
\end{align*}
Then Theorem~\ref{thm:visible_ft} implies that there exists $n_k \rightarrow \infty$, $T_k \in \Rb$, and $z_0 \in \Omega$ such that $\gamma_{n_k}(T_k) \rightarrow z_0 \in \Omega$. Then there exists $t_k \in [0,T]$ and $m_k \in \Zb$ such that 
\begin{align*}
s_{n_k} h^{m_k} \gamma(t_k) = \gamma_{n_k}(T_k).
\end{align*}
Then 
\begin{align*}
\lim_{k \rightarrow \infty} s_{n_k} h^{m_k} \gamma(t_k)= z_0.
\end{align*}
 Since $\Aut(\Omega)$ acts properly on $\Omega$, we can pass to another subsequence such that $s_{n_k} h^{m_k}$ converges in $\Aut(\Omega)$. Since $s_n \in S$ was an arbitrary sequence, we then see that the quotient $S/\{ h^n : n \in \Zb\}$ is compact. Thus we are in case (2) of the theorem. \\ 
 
\noindent \textbf{Case C:} $S_{m+1}$ is compact and $S_m$ contains a parabolic element $u \in S_m$. Arguing as in Case B, one can show that $S \ell^+_u =\ell^+_u$. Thus we are in case (3) of the theorem. \newline

We now prove the ``further'' part of the proof. Let $N$ be a connected subgroup that normalizes $S$. 

First suppose that there exists a term $S_{m+1}$ of the derived series of $S$ such that every element of $S_{m+1}$ is elliptic, $S_{m+1}$ is non-compact, and $\Lc(\Omega;S_{m+1})=\{x_0\}$. If $N$ normalizes $S$, then $N$ also normalizes $S_{m+1}$. Thus $Nx_0 = x_0$ by Proposition~\ref{prop:subgroups}. 

Next suppose that $S$ contains a hyperbolic element $h$ such that $S$ preserves the set $\{\ell^+_h, \ell^-_h\}$. If $n \in N$, then $n h n^{-1}$ is hyperbolic with attracting/repelling fixed points $n\ell^\pm_h$. Since $nhn^{-1} \in S$, we also have that 
\begin{align*}
n h n^{-1}\{\ell^+_h, \ell^-_h\} = \{\ell^+_h, \ell^-_h\}.
\end{align*}
So by Corollary~\ref{cor:NS}, we must have $\{ n\ell^+_h, n\ell^-_h\} = \{\ell^+_h, \ell^-_h\}$. Since $N$ is connected, we then have $n \ell^\pm_h = \ell^\pm_h$ for all $n \in N$. 

Finally, suppose that $S$ contains a parabolic element $u$ and $S$ fixes $\ell^+_u$. Then arguing as in the previous case shows that $n \ell^+_u = \ell^+_u$ for all $n \in N$. 

\end{proof}

\section{Proof of Theorem~\ref{thm:main_ft}}\label{sec:proof_ft}

For the rest of this section, suppose that $\Omega$ is a bounded pseudoconvex domain with $C^\infty$ boundary and finite type. Further assume that $\Lc(\Omega)$ contains at least two points. 

\subsection{Constructing the group $G$}

\begin{lemma}\label{lem:non_compact} With the notation above, $\Omega$ is biholomorphic to a weighted homogeneous polynomial domain. In particular, $\Aut_0(\Omega)$ is non-compact and $\Aut_0(\Omega)$ acts without fixed points on $\partial \Omega$. \end{lemma}

\begin{proof} By Proposition~\ref{prop:construct_hyp}, $\Aut(\Omega)$ contains a hyperbolic element. Then by S.Y. Kim's rigidity result, see Theorem~\ref{thm:SYkim}, $\Omega$ is biholomorphic to a weighted homogeneous polynomial domain. Then Theorem~\ref{thm:SYKim2} implies that $\Aut_0(\Omega)$ is non-compact and acts without fixed points on $\partial \Omega$. \end{proof}

Let $G^{sol} \leq \Aut_0(\Omega)$ be the solvable radical of $\Aut_0(\Omega)$, that is let $G^{sol}$ be the maximal connected, closed, normal, solvable subgroup of $\Aut_0(\Omega)$. Notice that $G^{sol}$ is also a normal subgroup of $\Aut(\Omega)$. Next let $G^{ss} \leq \Aut_0(\Omega)$ be a semisimple subgroup such that $\Aut_0(\Omega) = G^{ss} G^{sol}$ is a Levi-Malcev decomposition of $\Aut_0(\Omega)$.

\begin{lemma} With the notation above, $G^{sol}$ is compact. In particular, $G^{ss}$ is non-compact. \end{lemma}

\begin{proof} If $G^{sol}$ were non-compact, then the ``further'' part of Theorem~\ref{thm:solvable_subgroups} would imply that $\Aut_0(\Omega)$ fixes a point in $\partial \Omega$ which is impossible by the last lemma. 
\end{proof}

\begin{lemma}\label{lem:torus_ft} With the notation above, $G^{sol}$ is a torus and $G^{sol}$ is in the center of $\Aut_0(\Omega)$. \end{lemma}

\begin{proof} First, note that $G^{sol}$ is a torus (every compact, connected, solvable Lie group it is isomorphic to a torus). Then since $G^{sol}$ is normal in $\Aut(\Omega)$, every $g \in \Aut(\Omega)$ induces an automorphism $\tau: \Aut(\Omega) \rightarrow \Aut(G^{sol})$ defined by $\tau(g)(h) = g h g^{-1}$.  Since $G^{sol}$ is a torus, $\Aut(G^{sol})$ is isomorphic to $\GL_n(\Zb)$ for some $n$. Since $\Aut_0(\Omega)$ is connected, we then see that $\Aut_0(\Omega)\leq \ker \tau$ and hence $G^{sol}$ is in the center of $\Aut_0(\Omega)$. \end{proof}

\begin{remark} Lemma~\ref{lem:torus_ft} implies that $\Aut_0(\Omega)$ is a reductive group, which immediately implies the next two lemmas. But to minimize the amount of Lie theory required we give direct proofs.  \end{remark}

\begin{lemma} With the notation above, $G^{ss}$ is a normal subgroup in $\Aut(\Omega)$. \end{lemma}

\begin{proof} If $g \in \Aut(\Omega)$, then 
\begin{align*}
gG^{ss}g^{-1}G^{sol} = gG^{ss}g^{-1}gG^{sol}g^{-1}=gG^{ss} G^{sol} g^{-1} = g \Aut_0(\Omega) g^{-1} = \Aut_0(\Omega)
\end{align*}
since $G^{sol}$ and $\Aut_0(\Omega)$ are normal subgroups of $\Aut(\Omega)$. So $gG^{ss}g^{-1}$ is a Levi factor of $\Aut_0(\Omega)$. Since every two Levi factors are conjugate (see~\cite[Chapter 6, Theorem 3]{OV1990}), there exists some $h \in \Aut_0(\Omega)$ such that $hG^{ss}h^{-1} = gG^{ss} g^{-1}$. But then $h=h_1s$ for $h_1 \in G^{ss}$ and $s \in G^{sol}$. Then since $G^{sol}$ is in the center of $\Aut_0(\Omega)$, we see that 
\begin{equation*}
gG^{ss} g^{-1} = hG^{ss} h^{-1} = h_1G^{ss}h_1^{-1}=G^{ss}. \qedhere
\end{equation*}
\end{proof}

As in Section~\ref{sec:basic_properties} we can write $G^{ss}$ as an almost direct product $G_1,\dots, G_m$ where each $G_i$ is a closed simple Lie subgroup of $G^{ss}$. Then define 
\begin{align*}
G:=\prod\{ G_i : G_i \text{ is non-compact}\}.
\end{align*}

\begin{lemma} With the notation above, $G$ is a normal subgroup of $\Aut(\Omega)$. \end{lemma} 

\begin{proof} Since $G^{ss}$ is a normal subgroup of $\Aut(\Omega)$, any $g \in \Aut(\Omega)$ induces an automorphism $C_g : G^{ss} \rightarrow G^{ss}$ defined by $C_g(h) = ghg^{-1}$. Next let $\gL$ be the Lie algebra of $G$. Then
\begin{align*}
\gL = \gL_1 \oplus \dots \oplus \gL_m
\end{align*}
where $\gL_i$ is the Lie subalgebra of $G_i$ (see the discussion in Section~\ref{sec:basic_properties}). Let $\pi_i : \gL \rightarrow \gL_i$ denote the natural projection. Now fix some $G_j$ and some $g \in \Aut(\Omega)$. Consider the induced map $\pi_i \circ d(C_g)|_{\gL_j} : \gL_j \rightarrow \gL_i$. Since $\gL_j$ is a simple Lie algebra,  $ \pi_i \circ d(C_g) |_{\gL_j}$ is either injective or the zero map. For the same reason, $ \pi_j \circ d(C_g^{-1}) |_{\gL_i}$ is either injective or the zero map. Thus $ \pi_i \circ d(C_g) |_{\gL_j}$ is either an isomorphism or the zero map. So 
\begin{align*}
d(C_g)(\gL_j) \subset \oplus \{ \gL_i : \gL_i \text{ is isomorphic to } \gL_j\}.
\end{align*}

However, $G_j$ is compact if and only if the Killing form of $\gL_j$ is a negative definite bilinear form~\cite[Chapter IV, Proposition 4.27]{K2002}. This implies that when $G_i$ is non-compact we have $C_g(G_i) \leq G$. So $G$ is a normal subgroup of $\Aut(\Omega)$.

\end{proof}

\begin{lemma}\label{lem:limit_sets_same_ft} With the notation above, $G$ contains a hyperbolic element and $\Lc(\Omega; G) = \Lc(\Omega)$.  \end{lemma}

\begin{proof} Since $G^{ss}$ is non-compact, $G$ is also non-compact and so $\Lc(\Omega; G)$ is non-empty. By Proposition~\ref{prop:subgroups}, $\Lc(\Omega; G)$ is closed and $\Aut(\Omega)$-invariant. By Lemma~\ref{lem:non_compact}, $\Aut(\Omega)$ has no fixed points in $\partial \Omega$, so $\Lc(\Omega;G)$ contains at least two points. So $G$ contains a hyperbolic element by Proposition~\ref{prop:construct_hyp}. 

Now fix some $x \in \Lc(\Omega)$. Then there exists $z_0 \in \Omega$ and $\varphi_n \in \Aut(\Omega)$ such that $\varphi_n(z_0) \rightarrow x$. By passing to a subsequence we can suppose that $\varphi_n^{-1}(z_0) \rightarrow y \in \partial \Omega$. Then by Theorem~\ref{thm:Bell2}, $\varphi_n(z) \rightarrow x$ for all $z \in \overline{\Omega} \setminus \{y\}$. Since $\Lc(\Omega; G)$ is not a single point, there exists some $z \in \Lc(\Omega; G)$ such that $z \neq y$. Then $\varphi_n(z) \rightarrow x$. By Lemma~\ref{prop:subgroups},  $\Lc(\Omega; G)$ is closed and $\Aut(\Omega)$-invariant, so we see that $x \in \Lc(\Omega; G)$. Since $x \in \Lc(\Omega)$ was arbitrary, we see that $\Lc(\Omega; G) = \Lc(\Omega)$. 
\end{proof}

\begin{lemma}\label{lem:G_has_no_fixed_pts} With the notation above, $G$ acts without fixed points on $\partial \Omega$. \end{lemma}

\begin{proof} 
Let 
\begin{align*}
N_0:=\prod\{ G_i : G_i \text{ is compact}\}.
\end{align*} 
Then define $N_1 := N_0 G^{sol}$. Then, by construction, $\Aut_0(\Omega)$ is the almost direct product of $G$ and $N_1$. 

Now suppose that $x_0 \in \partial \Omega$. We claim that there exists $g \in G$ such that $gx_0 \neq x_0$. By Theorem~\ref{thm:SYKim2} Part (3) there exists a hyperbolic element $h \in \Aut_0(\Omega)$ such that 
\begin{align*}
\ell^+_h, \ell^-_h, x_0
\end{align*}
are pairwise distinct. Now $h=gk$ for some $g \in G$ and $k \in N_1$. Fix some $z_0\in \Omega$. Since $N_1$ is compact the set $\{ k^n(z_0) : n \in \Zb\}$ is relatively compact in $\Omega$. So Theorem~\ref{thm:Bell2} implies that 
\begin{align*}
\lim_{n \rightarrow\pm \infty} g^n(z_0) = \lim_{n \rightarrow \pm\infty} h^n(k^{-n}z_0) = \ell^\pm_h.
\end{align*}
So $g$ is hyperbolic and $\ell^\pm_g = \ell^\pm_h$. Since $x_0 \neq \ell^-_h$, Theorem~\ref{thm:Bell2} implies that $g^nx_0 \rightarrow \ell^+_h$. So $gx_0 \neq x_0$. 
\end{proof}

\subsection{Real rank one and finite center}

In this subsection we will show that $G$ is a simple Lie group with real rank one and finite center.

Given $g \in G$, let $C(g)$ denote the centralizer of $g$ in $G$. 

\begin{lemma}\label{lem:finite_centralizers} With the notation above, if $h \in G$ is hyperbolic, then the quotient $C(h) / \{ h^n : n \in \Zb\}$ is compact. \end{lemma}

\begin{proof}
Fix a sequence $g_n \in C(h)$. We claim that there exists $n_k \rightarrow \infty$ and a sequence $m_k \in\Zb$ such that $g_{n_k}h^{m_k}$ converges.

By Proposition~\ref{prop:translate_ft} there exists $\lambda \geq1$, $\kappa \geq 0$, $T >0$, and an $(\lambda, \kappa)$-almost-geodesic $\sigma : \Rb \rightarrow \Omega$ such that 
\begin{align*}
h^m\sigma(t) = \sigma(t+mT)
\end{align*}
for all $t \in \Rb$ and $m \in \Zb$. Since the set $\sigma([0,T]) \subset \Omega$ is compact, Theorem~\ref{thm:Bell2} implies that 
\begin{align*}
\lim_{t \rightarrow \pm \infty} \sigma(t)=\ell^\pm_h.
\end{align*}

Consider the almost-geodesics $\sigma_n = g_n \sigma$. Since 
\begin{align*}
g_n \ell^\pm_h = \ell^\pm_{g_nhg_n^{-1}} = \ell^\pm_h
\end{align*}
 we then have that 
\begin{align*}
\lim_{t \rightarrow \pm \infty} \sigma_n(t) = \ell^\pm_h
\end{align*}
for all $n$. 

Then Theorem~\ref{thm:visible_ft} implies that there exists $n_k \rightarrow \infty$, $T_k \in \Rb$, and $z_0 \in \Omega$ such that $\sigma_{n_k}(T_k) \rightarrow z_0 \in \Omega$. Then we can find $m_k \in \Zb$ and $t_k \in [0,T]$ such that 
\begin{align*}
\sigma_{n_k}(T_k) =  g_{n_k} h^{m_k}\sigma(t_k).
\end{align*}
Then \begin{align*}
\lim_{k \rightarrow \infty} g_{n_k} h^{m_k}\sigma(t_k) = z_0.
\end{align*}
Since $\Aut(\Omega)$ acts properly on $\Omega$, we can pass to another subsequence such that $g_{n_k} h^{m_k}$ converges in $\Aut(\Omega)$. Since $g_n$ was an arbitrary sequence in $C(h)$ we then see that $C(h) / \{ h^n : n \in \Zb\}$ is compact.
\end{proof}

\begin{lemma} With the notation above, $G$ has finite center. \end{lemma}

\begin{proof} Since $G$ is semisimple, the center of $G$ is discrete. So this follows immediately from Lemma~\ref{lem:finite_centralizers}. \end{proof}

Fix a norm on $\gL$, the Lie algebra of $G$, and let $\norm{\cdot}$ be the associated operator norm on $\SL(\gL)$. 

\begin{lemma}\label{lem:dist_norm_est} With the notation above, if $z_0 \in \Omega$, then there exists some $\alpha \geq 1$ and $\beta \geq 0$ such that 
\begin{align*}
K_\Omega( g(z_0), z_0) \leq \alpha \log \norm{\Ad(g)} + \beta
\end{align*}
for all $g \in G$. 
\end{lemma}

\begin{proof}
By Theorem~\ref{ref:iwasawa}, there exists a compact subgroup $K \leq G$ and a connected abelian subgroup $A \leq G$ such that $G = K A K$ and $\Ad(A)$ is diagonalizable in $\SL(\gL)$. Let $\aL$ be the Lie algebra of $A$. Since $A$ is abelian and connected the map $X \in \aL \rightarrow \exp(X) \in A$ is an Lie group isomorphism. Next let $\norm{ \cdot }_{\aL}$ be a norm on $\aL$. Since $\Ad(A)$ is diagonalizable in $\SL(\gL)$ there exists $\alpha_0 \geq 1$ such that 
\begin{align*}
\frac{1}{\alpha_0} \norm{X}_{\aL} \leq \log \norm{\Ad(e^X)} \leq \alpha_0 \norm{X}_{\aL}.
\end{align*}

Since the action on $\Aut(\Omega)$ on $\Omega$ is smooth, there exists an $C > 0$ such that 
\begin{align*}
K_\Omega(e^X z_0, z_0) \leq C\norm{X}_{\aL}
\end{align*}
for all $X \in \aL$ with $\norm{X}_{\aL} \leq 1$. 
Then if $X \in \aL$ let $X_0 = X/\norm{X}_{\aL}$ and $n = \lfloor \norm{X}_{\aL} \rfloor$. Then 
\begin{align*}
K_\Omega(e^X z_0, z_0) 
&\leq K_\Omega(e^{X} z_0, e^{nX_0}z_0) + K_\Omega(e^{n X_0} z_0, e^{(n-1)X_0}z_0)+ \dots + K_\Omega(e^{X_0} z_0, z_0) \\
& \leq C \norm{X-nX_0}_{\aL} + C\norm{X_0}_{\aL} + \dots + C\norm{X_0}_{\aL} = C\norm{X}_{\aL}.
\end{align*}

Further, since $K$ is compact, there exists some $R \geq 0$ such that 
\begin{align*}
K_\Omega(k(z_0), z_0) \leq R
\end{align*}
for all $k \in K$. By increasing $R$ if necessary, we can further assume that 
\begin{align*}
\log \norm{\Ad(k)} \leq R
\end{align*}
for all $k \in K$. 

Now suppose that $g \in G$. Then $g = k_1 e^X k_2$ for some $X \in \aL$ and $k_1,k_2 \in K$. Then
\begin{align*}
K_\Omega( g(z_0), z_0) 
&\leq 2R + K_\Omega(e^Xz_0,z_0) \leq 2R + C\norm{X}_{\aL} \leq 2R+\alpha_0C\log\norm{\Ad(e^X)}.
\end{align*}
Further 
\begin{align*}
\log\norm{\Ad(e^X)} \leq \log\norm{\Ad(k_1^{-1})}\norm{\Ad(g)}\norm{\Ad(k_2^{-1})} \leq 2R+\log\norm{\Ad(g)}.
\end{align*}
So
\begin{equation*}
K_\Omega( g(z_0), z_0)  \leq 2R(1+ \alpha_0C) + \alpha_0C\log \norm{\Ad(g)}. \qedhere
\end{equation*}
\end{proof}

\begin{definition} An element $g \in G$ is \emph{$\Lc$-hyperbolic} (respectively \emph{$\Lc$-elliptic}, \emph{$\Lc$-unipotent}) if $g$ is hyperbolic (respectively elliptic, unipotent)  in $G$ in the Lie group sense (see Section~\ref{sec:basic_properties}).
\end{definition}

\begin{lemma}\label{lem:double_hyp} With the notation above, there exists an element $g \in G$ which is both hyperbolic and $\Lc$-hyperbolic. \end{lemma}

\begin{proof}
By Lemma~\ref{lem:limit_sets_same_ft} there exists some $g \in G$ which is hyperbolic. Then by Lemma~\ref{lem:trans_dist_ft} 
\begin{align*}
\lim_{n \rightarrow \infty} \frac{K_\Omega(g^n(z), z) }{n} >0
\end{align*}
for all $z \in \Omega$. So by Lemma~\ref{lem:dist_norm_est} 
\begin{align}
\label{eq:ineq_norm_growth}
\liminf_{n \rightarrow \infty} \frac{\log \norm{\Ad(g)^n}}{n}  >0.
\end{align}
Using the Jordan decomposition, see Theorem~\ref{thm:jordan_decomp}, we can write $g = khu$ where $k$ is $\Lc$-elliptic, $h$ is $\Lc$-hyperbolic, $u$ is $\Lc$-unipotent, and $k,h, u$ commute. Then since $\Ad(k)$ is elliptic and $\Ad(u)$ is unipotent we have
\begin{align*}
 \lim_{n \rightarrow \infty} \frac{\log \norm{\Ad(ku)^n}}{n} \leq  \lim_{n \rightarrow \infty} \frac{\log \norm{\Ad(k)^n}+\log \norm{\Ad(u)^n}}{n} =0
\end{align*}
so
\begin{align*}
 \liminf_{n \rightarrow \infty} \frac{\log \norm{\Ad(h)^n}}{n}  \geq \liminf_{n \rightarrow \infty} \frac{\log \norm{\Ad(g)^n}-\log \norm{\Ad(ku)^n}}{n} >0.
\end{align*}
Thus $\Ad(h) \neq 1$. 

We claim that $ku$ is elliptic (as in Definition~\ref{defn:elems_ft}). Since 
\begin{align*}
\lim_{n \rightarrow \infty} \frac{\log \norm{\Ad(ku)^n}}{n}  =0
\end{align*}
Lemmas~\ref{lem:trans_dist_ft} and~\ref{lem:dist_norm_est} imply that $ku$ is not hyperbolic. Now fix some $z_0 \in \Omega$. Since $ku$ commutes with $g$ we see that 
\begin{align*}
ku(\ell^{\pm}_g) = ku \left( \lim_{n \rightarrow \pm \infty} g^n z_0 \right) = \lim_{n \rightarrow \pm \infty} g^n kuz_0 = \ell^{\pm}_g.
\end{align*}
So $ku$ cannot be parabolic by Corollary~\ref{cor:parabolic_dynamics}. So $ku$ must be elliptic. 

Now since $ku$ is elliptic, the set $\{ (ku)^nz_0 : n \in \Zb\}$ is relatively compact in $\Omega$. So by Corollary~\ref{cor:NS} 
\begin{align*}
\lim_{n \rightarrow \pm \infty} h^n(z_0) = \lim_{n \rightarrow \pm \infty} g^n((ku)^{-n} z_0) = \ell^{\pm}_g.
\end{align*}
So $h$ is hyperbolic.
\end{proof}

\begin{lemma} With the notation above, $G$ is a simple Lie group of non-compact type and has real rank one. \end{lemma}

\begin{proof} Pick an element $h \in G$ which is hyperbolic and $\Lc$-hyperbolic. By Proposition~\ref{prop:cartan_conj}, there exists a maximal Cartan subgroup $A \leq G$ such that $h \in Z(G)A$. Then $Z(G)A \leq C(h)$ and so by Lemma~\ref{lem:finite_centralizers} the quotient $Z(G)A / \{ h^n : n \in \Zb\}$ is compact. Since $A$ is isomorphic to $\Rb^r$ where $r = { \rm rank}_{\Rb}(G)$, this implies that $r=1$.\end{proof}

\subsection{The automorphism group has finitely many components}\label{sec:finitely_comp}

In this section we show that $\Aut_0(\Omega)$ has finite index in $\Aut(\Omega)$.

Since $G$ is a normal subgroup in $\Aut(\Omega)$, associated to every $g \in \Aut(\Omega)$ is an element $\tau(g) \in { \rm Aut}(G)$ defined by 
\begin{align*}
\tau(g)( h) = ghg^{-1}.
\end{align*}
Next let ${ \rm Inn}(G)$ denote the \emph{inner automorphisms of $G$}, that is the automorphisms of the form $g \rightarrow hgh^{-1}$ where $h \in G$. Then let ${\rm Out}(G) = \Aut(G)/{\rm Inn}(G)$. Finally define $[\tau]:\Aut(\Omega) \rightarrow {\rm Out}(G)$ by letting $[\tau](g)$ denote the equivalence class of $\tau(g)$. 

Since $G$ is semisimple, ${ \rm Out }(G)$ is finite (see for instance~\cite[Chapter X]{H2001}). So to prove that $\Aut_0(\Omega)$ has finite index in $\Aut(\Omega)$, it is enough to prove the following. 

\begin{lemma}\label{lem:finite_comp}
With the notation above, $\Aut_0(\Omega)$ has finite index in $\ker [\tau]$. In particular, $\Aut_0(\Omega)$ has finite index in $\Aut(\Omega)$.
\end{lemma}

\begin{proof} It is enough to show that the quotient $\ker [\tau] / G$ is compact. So suppose that $g_n \in \ker [\tau]$ is a sequence. We claim that there exists $n_k \rightarrow \infty$ and $h_k \in G$ such that $g_{n_k} h_k$ converges in $\Aut(\Omega)$. Now for each $n \in \Nb$ there exists some $\overline{g}_n \in G$ such that $\tau(g_n) = \tau(\overline{g}_n)$. Then by replacing each $g_n$ with $g_n\overline{g}_n^{-1}$ we can assume that 
\begin{align*}
g_n g g_n^{-1} = g
\end{align*} 
for every $g \in G$ and $n \in \Nb$. Now fix a hyperbolic element $h \in G$. Then $g_n \in C(h)$ and so by Lemma~\ref{lem:finite_centralizers}  there exists $n_k \rightarrow \infty$ and $m_k \in \Zb$ such that $g_{n_k} h^{m_k}$ converges in $\Aut(\Omega)$. Since $g_n$ was an arbitrary sequence in $\ker [\tau]$ we see that $\ker [\tau] / G$ is compact. Hence $\Aut_0(\Omega)$ has finite index in $\ker [\tau]$.

\end{proof} 

\subsection{The limit set is a sphere}\label{subsec:sphere_ft}

In this subsection we show that $\Lc(\Omega)$ is homeomorphic to a sphere. 

%
%

We now consider the symmetric space associated to $G$, see Section~\ref{sec:basic_properties} for more details. Let $K \leq G$ be a maximal compact subgroup and let $X = G/K$ be the associated symmetric space. Since $G$ has real rank one, $X$ is negatively curved. Let $X(\infty)$ be the geodesic boundary of $X$. Fix a point $\xi_0 \in X(\infty)$ and let $P$ be the stabilizier of $\xi_0$ in $G$. Since $G$ acts transitively on $X(\infty)$, see Section~\ref{sec:parabolic}, we can identify $X(\infty)$ with $G/P$. 

\begin{lemma}\label{lem:parabolics_ft} With the notation above, there exists a point $x_0 \in \Lc(\Omega)$ such that 
\begin{align*}
P = \{ g \in G : g(x_0) = x_0\}.
\end{align*}
Further, $G \cdot x_0 = \Lc(\Omega)$ and $\Lc(\Omega)$ is a smooth submanifold of $\partial \Omega$ diffeomorphic to a sphere of $\dim X-1$. \end{lemma}

\begin{proof}
Since $G$ acts transitively on $X(\infty)$, there exists an $\Lc$-hyperbolic element $h$ such that $\omega_h^+ = \xi_0$. 
Then by Theorem~\ref{thm:parabolic_decomp} the limit
\begin{align*}
\lim_{n \rightarrow \infty} h^{-n} p h^n \in G
\end{align*}
exists for every $p \in P$. 

Let $x_0 = \ell_h^+$. Then if $p \in P$ and $z \in \Omega$ we have
\begin{align*}
p x_0 = p \left(\lim_{n \rightarrow \infty} h^n z \right)=  \lim_{n \rightarrow \infty} h^n \left(h^{-n} p h^n\right) z = x_0
\end{align*}
by Theorem~\ref{thm:Bell2}. So $P$ fixes $x_0$. 

Let 
\begin{align*}
H = \{ g \in G : g(x_0) = x_0\}.
\end{align*}
Then $H$ is closed and $P \leq H$. So by Theorem~\ref{thm:parabolic} either $H=P$ or $H=G$. However Lemma~\ref{lem:G_has_no_fixed_pts} implies that $G \cdot x_0 \neq \{x_0\}$, so we must have that
\begin{align*}
P = \{ g \in G : g(x_0) = x_0\}.
\end{align*}

Then the map $g \in G/P \rightarrow g \cdot x_0$ induces a continuous, injective map  $G/P \rightarrow G \cdot x_0$. By the discussion in Section~\ref{sec:parabolic}, $G/P$ is diffeomorphic to a sphere of dimension $\dim X -1$. Then, since $G/P$ is compact, the map 
\begin{align*}
g \in G/P \rightarrow g \cdot x_0 \in G \cdot x_0
\end{align*}
 is actually a homeomorphism.  In particular, $G \cdot x_0$ is a compact subset of $\partial \Omega$. Since $G$ acts smoothly on $\partial \Omega$ and the orbit $G \cdot x_0$ is closed, it follows that $G \cdot x_0$ is a smooth submanifold of $\partial \Omega$ which is diffeomorphic to $G/P$, see for instance~\cite[Theorem 15.3.7]{tD2008}. 

We next show that $G \cdot x_0 = \Lc(\Omega)$. Suppose that $x \in \Lc(\Omega)$. By Lemma~\ref{lem:limit_sets_same_ft}, $\Lc(\Omega) = \Lc(\Omega; G)$. So there exist $z_0 \in \Omega$ and a sequence $g_n \in G$ such that $g_n(z_0) \rightarrow x$. By passing to a subsequence we can suppose that $g_n^{-1}(z_0) \rightarrow y$. Then by Theorem~\ref{thm:Bell2}, $g_n(z) \rightarrow x$ for all $z \in \overline{\Omega} \setminus \{ y\}$. Since $G \cdot x_0$ is not a single point, there exists some $g_0 \in G$ such that $g_0x_0 \neq y$. Then $g_n(g_0x_0) \rightarrow x$. Since $G \cdot x_0$ is compact we then see that $x \in G \cdot x_0$. 
\end{proof}

\subsection{The group $G$ is locally isomorphic to $\SU(1,k)$}\label{sec:finding_su}

In this subsection we prove that $G$ is locally isomorphic to $\SU(1,k)$ for some $k \geq 1$. 

If $\dim_{\Rb} X(\infty)=1$, then by the classification of negatively curved symmetric spaces $X$ must be isometric to real hyperbolic 2-space. Then $G$ is locally isomorphic to $\SU(1,1)$. 

Next assume that $\dim_{\Rb} X(\infty) \geq 2$. Then 
\begin{align*}
\dim_{\Rb} T_{z} \Lc(\Omega) + \dim_{\Rb} T_z^{\Cb} \partial \Omega \geq  2+(2d-2) = 2d
\end{align*}
so
\begin{align*}
T_{z} \Lc(\Omega) \cap T_z^{\Cb} \partial \Omega \neq (0)
\end{align*}
for every $z \in \Lc(\Omega)$. 

\begin{lemma} With the notation above, $T_{z} \Lc(\Omega)$ is not contained in $T_z^{\Cb} \partial \Omega$ for every $z \in \Lc(\Omega)$. In particular, 
\begin{align*}
\dim_{\Rb}\left(T_{z} \Lc(\Omega) \cap T_z^{\Cb} \partial \Omega \right)= \dim \Lc(\Omega)-1
\end{align*}
for all $z \in \Lc(\Omega)$.
\end{lemma}

\begin{proof} By Theorem~\ref{thm:SYKim2} part (2), there exists a point $z_0 \in \Lc(\Omega)$ such that  $T_{z_0} \Lc(\Omega)$ is not contained in $T_{z_0}^{\Cb} \partial \Omega$. Then since $G$ acts transitively on $\Lc(\Omega)$ we see that $T_{z} \Lc(\Omega)$ is not contained in $T_z^{\Cb} \partial \Omega$ for every $z \in \Lc(\Omega)$. 
\end{proof}

Then for $z \in \Lc(\Omega)$ let 
\begin{align*}
V_z = T_{z} \Lc(\Omega) \cap T_z^{\Cb} \partial \Omega.
\end{align*}
 Then $z \rightarrow V_z$ is a codimension one smooth distribution on $\Lc(\Omega)$. Further, since $G$ acts on $\Omega$ by biholomorphisms we see that $d(g)_zV_z = V_{gz}$ for all $g \in G$. So $V_z$ is a $G$-invariant distribution. So $G/P$ has a $G$-invariant codimension one smooth distribution. But this is only possible if $G$ is locally isomorphic to $\SU(1,k)$, see Theorem~\ref{thm:inv_dist}.

\subsection{Constructing an equivariant map}\label{subsec:equiv_maps} Recall that ${ \rm PU}(1,k)$ is the image of $\SU(1,k)$ in $\PGL_{k+1}(\Cb)$ and ${ \rm PU}(1,k)$ acts by fractional linear transformations on the unit ball $\Bb_k \subset \Cb^k$:
\begin{align*}
\begin{bmatrix} a & b^t \\ c & D \end{bmatrix} \cdot z = \frac{c+Dz}{a+b^tz}.
\end{align*}
In fact, this action gives an isomorphism 
\begin{align*}
\rho_0 : { \rm PU}(1,k) \rightarrow \Aut(\Bb_k)
\end{align*}

Since $G$ is locally isomorphic to $\SU(1,k)$, there exists an isomorphism $\pi: G/Z(G) \rightarrow { \rm PU}(1,k)$. So we have an isomorphism $\rho: G/Z(G) \rightarrow \Aut(\Bb_k)$ defined by $\rho = \rho_0 \circ \pi$. 

Now let $P$ be the group from Section~\ref{subsec:sphere_ft}. Then $\rho(P)$ is the stabilizer of a point in $w_0 \in \partial \Bb_k$. This follows from the fact that $\rho(P)$ is a parabolic subgroup of $\Aut(\Bb_k)$ or by simply repeating the proof of Lemma~\ref{lem:parabolics_ft} (since $\Bb_k$ is itself a bounded pseudoconvex domain with finite type). 

\begin{lemma}\label{lem:commute_fix} With the notation above, if $\varphi \in \Aut(\Omega)$ commutes with $G$, then $\varphi(x) =x$ for all $x \in \Lc(\Omega)$. \end{lemma}

\begin{proof}
By Lemma~\ref{lem:limit_sets_same_ft}, $\Lc(\Omega) = \Lc(\Omega; G)$. So if $x \in \Lc(\Omega)$ then there exists $z_0 \in \Omega$ and a sequence $g_n \in G$ such that $g_n(z_0) \rightarrow x$. Then 
\begin{align*}
\varphi(x)= \varphi\left( \lim_{n \rightarrow \infty} g_n(z_0) \right) = \lim_{n \rightarrow \infty}\varphi g_n(z_0) = \lim_{n \rightarrow \infty} g_n(\varphi z_0) = x
\end{align*}
by Theorem~\ref{thm:Bell2}.
\end{proof}

The above Lemma implies that $Z(G)$ acts trivially on $\Lc(\Omega)$ and so the action of $G$ on $\Lc(\Omega)$ induces an action of $G/Z(G)$ on $\Lc(\Omega)$. So we have a $\rho$-equivariant diffeomorphism $F: \Lc(\Omega) \rightarrow \partial \Bb_k$ defined by 
\begin{align*}
F(gx_0) = \rho(g) w_0.
\end{align*}

\subsection{The automorphism group is an almost direct product}\label{subsec:almost_dir_ft} In this section we prove that $\Aut(\Omega)$ is the almost direct product of $G$ and a compact subgroup, but first a lemma. 

\begin{lemma}\label{lem:compact_center} With the notation above, let $C$ denote the centralizer of $G$ in $\Aut(\Omega)$. Then $C$ is compact. 
\end{lemma}

\begin{proof} By Lemma~\ref{lem:commute_fix} each $c \in C$ acts trivially on $\Lc(\Omega)$. Since $\#\Lc(\Omega) > 2$,  Theorem~\ref{thm:Bell2} implies that $C$ is compact. 
\end{proof}

As in Section~\ref{sec:finitely_comp}, let $\tau: \Aut(\Omega) \rightarrow \Aut(G)$ denote the homomorphism given by $\tau(g)(h) = ghg^{-1}$. Notice that $\tau(g)(Z(G)) = Z(G)$ and so $\tau$ descends to an automorphism of $G/Z(G)$. Then $\tau$ induces a homomorphism $\Phi: \Aut(\Omega) \rightarrow \Aut({ \rm PU}(1,k))$ defined by 
\begin{align*}
\Phi(g) = \rho \circ \tau(g) \circ \rho^{-1}.
\end{align*}

Let $\theta : { \rm PU}(1,k) \rightarrow { \rm PU}(1,k)$ denote the automorphism 
\begin{align*}
\theta(g) = \overline{g}
\end{align*}
 and let ${ \rm Inn}(\PU(1,k))$ denote the automorphisms of the form $g \rightarrow hgh^{-1}$ where $h \in \PU(1,k)$. Then it is well known that 
\begin{align*}
\Aut(\PU(1,k)) = { \rm Inn}(\PU(1,k)) \cup { \rm Inn}(\PU(1,k)) \circ \theta.
\end{align*}

Finally define the subgroup 
\begin{align*}
N = \Phi^{-1}\left( \{ \id, \theta\} \right) \leq \Aut(\Omega).
\end{align*}

\begin{proposition} With the notation above, 
\begin{enumerate}
\item $N$ is a compact normal subgroup of $\Aut(\Omega)$,
\item $\Aut(\Omega)$ is the almost direct product of $G$ and $N$. 
\end{enumerate}
\end{proposition}

\begin{proof}
By definition $N$ is a normal closed subgroup of $\Aut(\Omega)$. Further, since $\pi: G/Z(G) \rightarrow \PU(1,k)$ is an isomorphism, we see that $G \cap N = Z(G)$. In particular, $G \cap N$ is finite.

We next claim that $GN = \Aut(\Omega)$. Consider some $g \in \Aut(\Omega)$. Then, since $\pi: G/Z(G) \rightarrow \PU(1,k)$ is an isomorphism, there exists some $h \in G$ such that $\Phi(hg) \in \{ \id, \theta\}$. So $g \in GN$. So $GN = \Aut(\Omega)$.

Next, we claim that $G$ and $N$ commute. Since $G$ and $N$ are normal subgroups we see that 
\begin{align*}
[G,N] \leq G \cap N = Z(G).
\end{align*}
But for $n \in N$ fixed, the set 
\begin{align*}
\left\{ ngn^{-1}g^{-1} : g \in G\right\} \leq [G,N] \leq Z(G)
\end{align*}
is connected and finite, so we see that $ngn^{-1} g^{-1} = 1$ for all $g \in G$. Since $n \in N$ was arbitrary, we then see that $ng = gn$ for all $n \in N$ and $g \in G$. 

Finally since $N$ is closed and commutes with $G$, Lemma~\ref{lem:compact_center} implies that $N$ is compact. 
\end{proof}

\section{Finite jet determination}\label{sec:finite_jet}

In this section we prove Corollary~\ref{cor:finite_jet} from the introduction. We will use the following two facts from Riemannian geometry. 

\begin{lemma}\label{lem:riem_1} Suppose $K$ is a compact Lie group acting smoothly on a compact manifold $M$. Then there exists a $K$-invariant Riemannian metric on $M$. 
\end{lemma}

\begin{proof}[Proof Sketch] Fix any Riemannian metric $g$ on $M$ and let $\mu$ be the Haar measure on $K$.  Then define a new Riemannian metric $\overline{g}$ by 
\begin{align*}
\overline{g}_x(v,w) = \int_{K} g_{kx}\left( d(k)_x v, d(k)_x w \right) d\mu(k).
\end{align*}
Then $\overline{g}$ is an $K$-invariant Riemannian metric on $M$. 
\end{proof}

\begin{lemma}\label{lem:riem_2} Suppose $(M,g)$ is a Riemannian manifold. If $F_1, F_2:M \rightarrow M$ are isometries and 
\begin{align*}
j_1(M, F_1, x) = j_1(M, F_2, x)
\end{align*}
for some $x \in M$, then there exists a neighborhood $U$ of $x$ such that $F_1|_U = F_2|_U$. 
\end{lemma}

\begin{remark} When $(M,g)$ is a complete Riemannian manifold, the conclusion of the lemma can be upgraded to say that $F_1 = F_2$. \end{remark}

\begin{proof}[Proof Sketch] For details see for instance~\cite[Chapter 1, Lemma 11.2]{H2001}. The idea is to find a neighborhood $V$ of $0$ in $T_xM$ where the exponential map $\exp_x:V \rightarrow M$ is well defined. Then prove that 
\begin{align*}
F(\exp_x(v)) = \exp_{F(x)}(dF_x(v))
\end{align*}
when $F:M \rightarrow M$ is an isometry and $v \in V$. Then let $U=\exp_x(V)$. 
 \end{proof}

We will also need the following basic fact about holomorphic maps.

\begin{lemma}\label{lem:fatou} Suppose $\Omega \subset \Cb^d$ is a bounded domain with $C^1$ boundary and $f : \Omega \rightarrow \Cb$ is a holomorphic map that extends continuously to $F:\partial\Omega \rightarrow \Cb$. If $F^{-1}(0)$ has non-empty interior in $\partial \Omega$, then $f$ is identically zero. 
\end{lemma}

\begin{proof} This is a simple consequence of the Luzin-Privalov theorem, see~\cite[Theorem 2.5]{CL1966}. \end{proof}

Now for the rest of the section, suppose that $\Omega \subset \Cb^d$ is a bounded pseudoconvex domain with $C^\infty$ boundary and finite type. Further assume that $\Lc(\Omega)$ contains at least two distinct points. Let $G$ and $N$ be the groups in Theorem~\ref{thm:main_ft}. 

\begin{lemma}\label{lem:finite_jets_1} With the notation above, for any $x \in \partial \Omega$ the map 
\begin{align*}
g \in N \rightarrow j_1(\partial\Omega, g,x) \in  {\rm Jet}_1(\Omega, x)
\end{align*}
is injective. 
\end{lemma}

\begin{proof} Since $N$ is a compact Lie group acting smoothly on $\partial \Omega$, this follows from Lemmas~\ref{lem:riem_1},~\ref{lem:riem_2}, and~\ref{lem:fatou}. \end{proof}

\begin{lemma} With the notation above, for any $x \in \Lc(\Omega)$ the map 
\begin{align*}
g \in \Aut(\Omega) \rightarrow j_2(\partial \Omega, g,x) \in  {\rm Jet}_2(\partial \Omega, x)
\end{align*}
is injective. 
\end{lemma}

\begin{proof} It is enough to show that: if $\varphi \in \Aut(\Omega)$ and 
\begin{align*}
j_2(\partial \Omega, \varphi,x)=j_2(\partial \Omega, \id,x),
\end{align*}
then $\varphi = \id$. Now $\varphi = gk$ for some $g \in G$ and $k \in N$. By Lemma~\ref{lem:compact_center}, $k(x) = x$ for all $x \in \Lc(\Omega)$ and so 
\begin{align*}
j_2(\Lc(\Omega), g,x)=j_2(\Lc(\Omega), \varphi,x)=j_2(\Lc(\Omega), \id,x).
\end{align*}
Now there exists an isomorphism $\rho: G/Z(G) \rightarrow \Aut(\Bb_k)$ and a $\rho$-equivariant diffeomorphism $F: \Lc(\Omega) \rightarrow \partial \Bb_k$. So 
\begin{align*}
j_2(\partial \Bb_k, \rho(g),F(x))=j_2(\partial \Bb_k, \id,F(x)).
\end{align*}
Which implies that $\rho(g) = \id$ and hence $g \in Z(G)$. Since $Z(G) \leq N$ (by the construction of $N$), we then have that $\varphi \in N$. So by Lemma~\ref{lem:finite_jets_1} we see that $\varphi = \id$.
\end{proof}

\begin{lemma} With the notation above, for any $x \in \partial \Omega \backslash \Lc(\Omega)$ the map
\begin{align*}
g \in \Aut(\Omega) \rightarrow j_1(\partial\Omega, g,x) \in  {\rm Jet}_1(\Omega, x)
\end{align*}
is injective. 
\end{lemma}

\begin{proof} Let $M:=\partial \Omega \setminus \Lc(\Omega)$, We first observe that $\Aut(\Omega)$ acts properly on $M$. To see this assume for a contradiction that $\varphi_n \rightarrow \infty$ in $\Aut(\Omega)$, but there exists a compact subset $K \subset \partial M$ such that $K \cap \varphi_n(K) \neq \emptyset$. Now fix some $z_0 \in \Omega$. By passing to a subsequence we can assume that $\varphi_n(z_0) \rightarrow x \in \partial \Omega$ and $\varphi_n(z_0) \rightarrow y$. But then by Theorem~\ref{thm:Bell2}, $\varphi_n(z)$ converges locally uniformly to $x$ on $\overline{\Omega} \setminus \{y\}$. Since $x,y \in \Lc(\Omega)$ we then see that $K \cap \varphi_n(K) = \emptyset$ for $n$ large. So we have a contradiction. 

Then by a result of Palais~\cite[Theorem 4.3.1]{P1961}, there exists a $\Aut(\Omega)$-invariant metric $g$ on $M$. Then the result follows from Lemmas~\ref{lem:riem_2} and~\ref{lem:fatou}.

\end{proof}

\section{Tits alternative}\label{sec:tits_alternative}
 
 In this section we prove Corollary~\ref{cor:tits} from the introduction.  We will reduce to the following variant of the Tits' alternative.
 
 \begin{theorem}[Tits~\cite{T1972}]\label{thm:tits} Suppose $G$ is a Lie group with finitely many components and $H \leq G$ is a subgroup. Then either $H$ contains a free group or has a finite index solvable subgroup.
 \end{theorem}

 For the rest of the section suppose that $\Omega$ is a bounded pseudoconvex domain with real analytic boundary and $H \leq \Aut(\Omega)$ is a subgroup. We claim that either $H$ contains a free group or a finite index solvable subgroup. Since every bounded pseudoconvex domain with real-analytic boundary is of finite type (see Remark~\ref{rmk:real_analytic}), we can apply Theorem~\ref{thm:main_ft}.
 
 Now $\Aut(\Omega)$ is a Lie group. If $\Aut(\Omega)$ is compact, then it has finitely many components. So we can apply Theorem~\ref{thm:tits}. If $\Aut(\Omega)$ is non-compact, then $\Lc(\Omega)$ is non-empty. If $\Lc(\Omega)$ contains at least two points, then $\Aut(\Omega)$ has finitely many components by Theorem~\ref{thm:main_ft}. So we can apply Theorem~\ref{thm:tits} again. 
 
 It remains to consider the case in which $\Lc(\Omega)=\{x_0\}$. Then $\Aut(\Omega)$ fixes  $x_0 \in \partial \Omega$. Let $J_k(\partial\Omega;x_0)$ denote the group of $k$-jets of smooth maps $f: \partial\Omega \rightarrow \partial \Omega$ with $f(x_0)=x_0$. Then by~\cite[Theorem 5]{BER2000}, there exists some $N$ and such that the induced homomorphism $\iota:\Aut(\Omega) \rightarrow J_N(\partial \Omega;x_0)$ is injective. Further, $J_N(\partial \Omega; x_0)$ is a Lie group with finitely many components so we can apply Theorem~\ref{thm:tits} again.

\appendix 

 \section{Semisimple Lie groups and symmetric spaces}\label{sec:basic_properties}

In the proof of Theorem~\ref{thm:main_ft}, we use some basic properties about semisimple Lie groups and the symmetric spaces they act on. In this section we recall these properties and give references. 

For the rest of the section we make the following assumption.

\begin{assumption} $G$ is a connected semisimple Lie group with finite center. \end{assumption}

Let $\gL$ be the Lie algebra of $G$. Then there is a Lie algebra decomposition 
\begin{align*}
\gL=\gL_1 \oplus \dots \oplus \gL_n
\end{align*}
into simple Lie subalgebras, see for instance~\cite[Chapter 1, Theorem 1.54]{K2002}. Then let $G_i$ be the connected subgroup of $G$ generated by $\exp(\gL_i)$. 

\begin{lemma} Each  $G_i$ is a closed subgroup of $G$ and $G$ is the almost direct product of $G_1,\dots, G_n$. \end{lemma}

\begin{proof} This is a well known fact, but here is a proof. By the Campbell-Baker-Hausdorff formula (see~\cite[Appendix B, Section 4]{K2002}) distinct pairs of $G_1,\dots, G_n$ commute. So distinct pairs of $G_1,\dots, G_n$ have intersection in $Z(G)$ and hence have finite intersection. The Campbell-Baker-Hausdorff formula also implies that the map 
\begin{align*}
(X_1,\dots, X_n) \in \gL_1 \oplus \dots \oplus \gL_n \rightarrow \exp(X_1)\exp(X_2)\dots\exp(X_n) \in G
\end{align*}
is a local diffeomorphism at $0$. So the product $G_1 \cdots G_n$ contains a open neighborhood of $\id$ in $G$. Since $G_1 \cdots G_n$ is a connected subgroup, this implies that $G=G_1 \cdots G_n$. Thus $G$ is the almost direct product of $G_1,\dots, G_n$.

Next we show that each $G_i$ is closed. Suppose that $g \in \overline{G_1}$. Then $g=g_1\dots g_n$ for some $g_j \in G_j$. So $g_1^{-1} g \in \overline{G_1} \cap  (G_2 \cdots G_n)$.  Thus $g_1^{-1} g \in Z(G)$. Since $g \in \overline{G_1}$ was arbitrary, we see that $\overline{G_1}\subset Z(G)G_1$ and in particular that $G_1$ has finite index in $\overline{G_1}$. Since $\overline{G_1}$ and $G_1$ are both connected, this implies that $G_1=\overline{G_1}$. Applying the same argument to the other factors shows that each $G_i$ is closed. 
\end{proof}

We now make an additional assumption:

\begin{addassumption} Every $G_i$ is non-compact. \end{addassumption}

Next let $\Ad: G \rightarrow SL(\gL)$ denote the adjoint representation. The kernel of $\Ad$ is the center of $G$, denoted $Z(G)$, so we have an isomorphism $G/Z(G) \cong \Ad(G)$. 

\begin{definition}
We then say an element $g \in G$ is:
\begin{enumerate}
\item \emph{semisimple} if $\Ad(g)$ is diagonalizable in $\SL(\gL^{\Cb})$,
\item \emph{hyperbolic} if $\Ad(g)$ is diagonalizable in $\SL(\gL)$ with all positive eigenvalues,
\item \emph{unipotent} if $\Ad(g)$ is unipotent in $\SL(\gL)$, and
\item \emph{elliptic} if $\Ad(g)$ is elliptic in $\SL(\gL)$.
\end{enumerate}
\end{definition}

Since $G$ is semisimple, every element can be decomposed into a product of a elliptic, hyperbolic, and unipotent element. More precisely:

\begin{theorem}[Jordan Decomposition]\label{thm:jordan_decomp} If $g \in G$, then there exists $g_e, g_h, g_u \in G$ such that 
\begin{enumerate}
\item $g=g_eg_hg_u$,
\item $g_e \in G$ is elliptic, $g_h \in G$ is hyperbolic, $g_u \in G$ is unipotent, and 
\item $g_e,g_h,g_u$ commute. 
\end{enumerate}
Moreover, the $g_e,g_h,g_u$ are unique up to factors in $\ker \Ad = Z(G)$.
\end{theorem}

\begin{proof} See for instance~\cite[Theorem 2.19.24]{E1996}. \end{proof}

A subgroup $A \leq G$ is called a \emph{Cartan subgroup} if $A$ is closed, connected, abelian, and every element in $A$ is hyperbolic. The \emph{real rank of $G$}, denoted by ${\rm rank}_{\Rb}(G)$, is defined to be
\begin{align*}
{\rm rank}_{\Rb}(G) = \max \{ \dim A: A \text{ is a Cartan subgroup of } \Ad(G)\}.
\end{align*}
We will need the following fact about Cartan subgroups.

\begin{proposition}\label{prop:cartan_conj} If $g \in G$ is hyperbolic and $A \leq G$ is a maximal Cartan subgroup, then $g$ is conjugate to an element of $Z(G)A$. 
\end{proposition}

\begin{proof} See for instance~\cite[Chapter IX, Theorem 7.2]{H2001}. \end{proof}

\begin{theorem}[Iwasawa Decomposition]\label{ref:iwasawa} If $A \leq G$ is a maximal Cartan subgroup, then there exists a compact subgroup $K \leq G$ such that $G = KAK$. \end{theorem}

We now focus on the real rank one case. 

\begin{addassumption} ${\rm rank}_{\Rb}(G)=1$. \end{addassumption}

Since 
\begin{align*}
{ \rm rank}_{\Rb}(G) = \sum_{i=1}^n { \rm rank}_{\Rb}(G_i)
\end{align*}
this implies that $G$ is a simple Lie group. In addition, by the classification of simple Lie groups, $G$ is locally isomorphic to one of $\SO(k,1)$, $\SU(k,1)$, $\Sp(k,1)$, or $F^{-20}_{4}$. 

Now fix $K \leq G$ a maximal compact subgroup. Then the quotient manifold $X=G/K$ is diffeomorphic to $\Rb^{\dim X}$ and has a unique (up to scaling) non-positively curved $G$-invariant Riemannian metric $g$, see~\cite[Section 2.2]{E1996} for details. Let $d_X$ denote the distance induced by $g$.

\begin{remark} Clearly $Z(G) \leq K$ and so $Z(G)$ acts trivially on $X$. For this reason, in many of the references cited in this section the group $G$ is assumed to have trivial center. 
\end{remark}

In the rank one case, the associated symmetric space $(X,d_X)$ is either a real hyperbolic space, a complex hyperbolic space, a quaternionic hyperbolic space, or the Cayley-hyperbolic plane. In all these cases, $(X,d_X)$ is a negatively curved Riemannian manifold. For details see~\cite[Chapter 19]{M1973}. 

Since $X$ is a non-positively curved simply connected Riemannian manifold, there exists a compactification called the \emph{geodesic compactification} which can be defined as follows. Let $\Gc$ denote the set of unit speed geodesic rays $\sigma:[0,\infty) \rightarrow X$. Then we say two geodesics $\sigma_1, \sigma_2 \in \Gc$ are equivalent if 
\begin{align*}
\lim_{t \rightarrow \infty} d_X(\sigma_1(t), \sigma_2(t)) < \infty.
\end{align*}
Finally let $X(\infty) = \Gc / \sim$. This gives a compactification $\overline{X} = X \cup X(\infty)$ of $X$ as follows. First fix a point $x_0 \in X$. Since $X$ is non-positively curved, for any $x \in X$ there exists a unique geodesic segment $\sigma_x$ joining $x_0$ to $x$. We then say that a sequence $x_n \in X$ converges to a point $\sigma \in X(\infty)$ if the geodesic segments $\sigma_{x_n}$ converge locally uniformly to $\sigma$. This construction does not depend on the initial choice of $x_0$. See~\cite[Section 1.7]{E1996} for details. 

Since $G$ acts by isometries on $X$ and the construction of $X(\infty)$ is independent of base point, the action of $G$ on $X$ extends to an action on $X \cup X(\infty)$. For a general non-positively curved simply connected Riemannian manifold this action is only continuous, but for negatively curved symmetric spaces we have the following.

\begin{theorem} With the notation above, $\overline{X}$ has a smooth structure, with this structure $X(\infty)$ is diffeomorphic to a sphere of dimension $\dim X -1$, and the action of $G$ on $X$ extends to a smooth action on $X(\infty)$. \end{theorem}

This theorem follows from considering the standard models of the negatively curved symmetric spaces, see~\cite[Chapter 19]{M1973}.

Although this will not be needed in the paper, it is worth observing the following fact about the action of hyperbolic elements. 

\begin{theorem} Suppose $h \in G$ is a hyperbolic element with $\Ad(h) \neq \id$. Then there exists distinct points $\omega_h^+, \omega_h^- \in X(\infty)$ such that $h(\omega^\pm_h) = \omega_h^\pm$. Further, If $U$ is a neighborhood of $\omega_h^+$ in $\overline{X}$ and $V$ is a neighborhood of $\omega_h^-$ in $\overline{X}$, then there exists some $N >0$ such that 
 \begin{align*}
h^n \left(\overline{X} \setminus V\right) \subset U \text{ and } h^{-n}\left(\overline{X} \setminus U\right) \subset V
 \end{align*}
 for all $n \geq N$.
\end{theorem}

\subsection{Parabolic subgroups}\label{sec:parabolic} A subgroup $P \leq G$ is called a \emph{parabolic subgroup of $G$} if $P$ is the stabilizer of some $\xi \in X(\infty)$. Since $G$ has real rank one, $G$ acts transitively on $X(\infty)$, see for instance~\cite[Proposition 2.21.13]{E1996}, and so there is a natural identification of $X(\infty)$ and $G/P$. So $G/P$ is diffeomorphic to a sphere of dimension $\dim X -1$. 

In the proof of Theorem~\ref{thm:main_ft}, we use the following fact about parabolic subgroups. 

\begin{theorem}\label{thm:parabolic}
With the notation above,  if $P \leq G$ is a parabolic subgroup, then $P$ is a maximal subgroup of $G$, that is: if $H$ is a closed subgroup of $G$ and $P \leq H$, then either $H=P$ or $H=G$.
\end{theorem}

\begin{proof} Suppose $P$ is the stabilizer of some $\xi \in X(\infty)$ and that $H$ is a closed subgroup with $P \lneqq H \leq G$. Then there exist some $h \in H$ with $h\xi=\eta$ and $\eta \neq \xi$. Then $hPh^{-1} \leq H$ is the stabilizer of $\eta$. Since $G$ has real rank one, $hPh^{-1}$ and $P$ are opposite parabolic subgroups and so $hPh^{-1}P$ is dense in $G$, see~\cite[Proposition 1.2.4.10]{W1972}. So $H=G$.
 \end{proof}
 
 \begin{theorem}\label{thm:parabolic_decomp} With the notation above, suppose $h \in G$ is hyperbolic and $P$ is the stabilizer of $\omega_h^+ \in X(\infty)$, then for every $p \in P$ the limit 
 \begin{align*}
 \lim_{n \rightarrow \infty} h^{-n} p h^n
 \end{align*}
exists in $G$. 
 \end{theorem}
 
 \begin{proof} See for instance~\cite[Proposition 2.17.3]{E1996}.\end{proof}
 
The action of $G$ on $G/P$ is very well understood and we have the following result about the existence of invariant distributions. 

\begin{theorem}\label{thm:inv_dist}
With the notation above, if $P \leq G$ is a parabolic subgroup and $G/P$ has a non-trivial $G$-invariant smooth distribution $V$, then either
\begin{enumerate}
\item $G$ is locally isomorphic to $\SU(1,k)$ and $V$ is a codimension one distribution,
\item $G$ is locally isomorphic to $\Sp(1,k)$ and $V$ is a codimension three distribution, or
\item $G$ is locally isomorphic to $F^{-20}_4$ and $V$ is a codimension seven distribution.
\end{enumerate}
\end{theorem}

\begin{proof} In each case there is an explicit model of the symmetric space $X$, see for instance~\cite[Chapter 19]{M1973}, and this result follows immediately from the considering these models.  \end{proof}

\bibliographystyle{alpha}
\bibliography{complex_kob}

\end{document}